\documentclass[10pt,reqno]{amsart}

 \usepackage{mathrsfs}

\usepackage{color}
\usepackage{amsmath, amsthm, amssymb}
\usepackage{amsfonts}
\usepackage[dvips]{epsfig}
\usepackage{graphicx}
\usepackage{caption}
\usepackage{subcaption}
\usepackage[english]{babel}
\usepackage{hyperref}

\usepackage{tikz}
\usepackage{rotating}
\usepackage[utf8]{inputenc}
\usepackage{cite}
\usepackage{amscd}
\usepackage{color}
\usepackage{bm}
\usepackage{enumerate}

\usepackage{verbatim}
\usepackage{hyperref}
\usepackage{amstext}
\usepackage{latexsym}
\theoremstyle{plain}
\newtheorem{theorem}{Theorem}[section]

\newtheorem{corollary}[theorem]{Corollary}
\newtheorem{proposition}[theorem]{Proposition}
\newtheorem{lemma}[theorem]{Lemma}

\numberwithin{theorem}{section}
\numberwithin{equation}{section}

\newcommand{\average}{{\mathchoice {\kern1ex\vcenter{\hrule height.4pt
width 6pt depth0pt} \kern-9.7pt} {\kern1ex\vcenter{\hrule
height.4pt width 4.3pt depth0pt} \kern-7pt} {} {} }}

\def\R{\mathbb{R}}

\renewcommand{\a }{\alpha }
\renewcommand{\b }{\beta }
\renewcommand{\d}{\delta }

\newcommand{\D }{\Delta }
\newcommand{\tr }{\hbox{ tr } }
\newcommand{\e }{\varepsilon }
\newcommand{\g }{\gamma}

\newcommand{\G }{\Gamma}
\renewcommand{\l }{\lambda }

\newcommand{\n }{\nabla }
\newcommand{\vp }{\varphi }
\renewcommand{\phi}{\varphi}

\newcommand{\rh }{\rho }

\newcommand{\s }{\sigma }

\renewcommand{\t }{\tau }

\renewcommand{\th }{\theta }

\renewcommand{\O }{\Omega }

\newcommand{\ov}{\overline}

\newcommand{\be}{\begin{equation}}
\newcommand{\ee}{\end{equation}}

\newcommand{\de}{\partial}

\newcommand{\ti}{\widetilde}

\renewcommand{\k}{\kappa}

\newcommand{\calO }{\mathcal{O}}
\newcommand{\calC }{\mathcal{C}}

\newcommand{\calD }{\mathcal{D}}

\newcommand{\N}{\mathbb{N}}

\newcommand{\cC}{{\mathcal C}}
\newcommand{\cD}{{\mathcal D}}

\newcommand{\cR}{{\mathcal R}}

\newcommand{\B}{{Q}}

\renewcommand{\epsilon}{\varepsilon}

\begin{document}
 
\title[Hardy-Sobolev inequality with singularity a curve]
{  Hardy-Sobolev inequality with singularity a curve}
\author{Mouhamed Moustapha Fall}
\address{M. M. F.: African Institute for Mathematical Sciences in Senegal, KM 2, Route de
Joal, B.P. 14 18. Mbour, Senegal.}
\email{mouhamed.m.fall@aims-senegal.org}

\author{El hadji Abdoulaye Thiam}
\address{E. H. A. T.:African Institute for Mathematical Sciences in Senegal, KM 2, Route de
Joal, B.P. 14 18. Mbour, Senegal. }
\email{elhadji@aims-senegal.org}
\begin{abstract}
We consider a     bounded domain $\O$ of $\R^N$, $N\ge3$, and $h$ a continuous function on $\O$. Let $\G$ be  a closed curve contained in $\O$.  We study existence of positive solutions $u \in H^1_0\left(\O\right)$ to the  equation
$$
-\D u+h u=\rho^{-\s}_\G u^{2^*_\s-1} \qquad \textrm{ in } \O
$$
where $2^*_\s:=\frac{2(N-\s)}{N-2}$, $\s\in (0,2)$, and $\rho_\G$ is the distance function to $\G$.  For $N\geq 4$,
 we find a sufficient condition, given by the local geometry of the curve,  for the existence of a ground-state solution. In the case $N=3$, we   obtain existence of ground-state solution provided  the trace of the regular part of the   Green of $-\D+h$ is positive at a point of the curve.   
\end{abstract}
\maketitle
\section{Introduction}\label{Intro}
For  $N\geq 3$, $0\leq k\leq N-1$ and $\s\in[0,2)$,     we consider   the Hardy-Sobolev  inequality
\begin{equation}\label{Hardy-Sobolev}
\int_{\R^N} |\nabla v|^2 dx \geq C  \biggl(\int_{\R^N} |z|^{-\s} |v|^{2^*_\s} dx\biggl)^{2/2^*_\s}\qquad \textrm{ for all $v \in \mathcal{D}^{1,2}({\R^N})$,}
\end{equation}
where $x=(t,z)\in \R^k \times \R^{N-k}$,  $C= C(N,\s,k)>0$ and $  2^*_\s:= \frac{2(N-\s)}{N-2}$. Here    the Sobolev space  $\calD^{1,2}(\R^N)$ is given by the completion of $C^\infty_c(\R^N)$ with respect to the norm $ v\longmapsto\left(\int_{\R^N} |\n v|^2 dx\right)^{1/2}.$
Inequality \eqref{Hardy-Sobolev} interpolates between cylindrical Hardy inequality, which corresponds to the case $\s=2$ and $k\not=N-2$, and the Sobolev inequality which is  the case $\s=0$. Moreover it is invariant under scaling  on $\R^N$ and by translations in the $t$-direction.
 It is well known   that in the case of Hardy inequality, $\s=2$ and $k\not=N-2$, there is no positive constant $C$ and $v\in \cD^{1,2}(\R^N)$ for which equality holds in  \eqref{Hardy-Sobolev}.  For $\s\in [0, 2)$, the best positive constant $C$ in \eqref{Hardy-Sobolev} is  
\be\label{eq:Hard-Sob-sharp}
S_{N,\s}:=\inf\left\{ \int_{\R^N} |\nabla v|^2 dx, \,\, v\in \cD^{1,2}(\R^N) \textrm{ and }\,  \int_{\R^N} |z|^{-\s} |v|^{2^*_\s} dx=1 \right\}.
%
\ee
In the case $\s=0$, $S_{N,0}$ is achieved by the standard bubble $(1+|x|^2)^{\frac{2-N}{2}}$, which is unique up to scaling 	and translations, see e.g. Aubin \cite{Talenti} and Talenti \cite{Aubin}.    For $k=0$, \eqref{Hardy-Sobolev} is a particular case of the Caffarelli-Kohn-Nirenberg inequality, see \cite{CKN}.  In this case,  Lieb  showed in \cite{Lieb} that  the  function   $(1+|x|^{2-\s})^{\frac{2-N}{2-\s}}$ achieves $S_{N,\s}$. When $k=N-1$, Musina proved in \cite{Musina} that the support of the minimizer is contained in a half-space. Therefore \eqref{Hardy-Sobolev} becomes the Hardy-Sobolev inequality with singularity all the boundary of the halfspace. \\
 For $1\leq k\leq N-2$ and $\s\in (0,2)$, Badiale and Tarentello proved the existence of a minimizer $w$ for \eqref{eq:Hard-Sob-sharp} in their paper \cite{BT}, where they were motivated by questions from astrophysics. Moreover   
Mancini, Fabri and Sandeep shwoed decay and symmetry properties of $w$ in    \cite{FMS}. In particular, they prove that $w(t,z)=\th(|t|, |z|) $, for some positive function $\th$.    An interesting classification result was also derived in \cite{FMS} when $\s=1$,  that every minimizer is of the form $ ((1+|z|)^2+|t|^2)^{\frac{2-N}{2}}$, up to scaling in $\R^N$ and translations in the $t$-direction. \\
Since in this paper we are interested with Hardy-Sobolev inequality with  weight singular at a given curve, our asymptotic energy level is given by $S_{N,\s}$  with     $k=1$ and $\s\in (0,2)$.\\

Let $\O$ be a bounded domain in $\R^N$, $N\geq 3$, and $h$ a continuous function on $\O$. Let  $\G \subset \O$ be a smooth closed curve. In this paper, we are concerned with the existence of minimizers for the infinimum
\be \label{eq:min-to-study}
\mu_{h}(\O,\G):=\inf_{ u\in  H^1_0(\O)  }\frac{ \displaystyle \int_{\O} |\nabla u|^2 dx +\int_\O h u^2 dx }{  \displaystyle \left(  \int_{\O} \rho_\G^{-\s} |u|^{2^*_\s} dx \right)^{\frac{2}{2^*_\s}}},
\ee
where $\s\in[0,2]$,  $\displaystyle 2^*_\s:= \frac{2(N-\s)}{N-2}$ and $ \rho_\G(x):=\textrm{dist}(x,\G)$.  Here and in the following, we assume that $-\D+h$ defines a  coercive bilinear form on $H^1_0(\O)$ ---which is a necessary condition for the existence of minimmizer for $\mu_{h}(\O,\G)$. We are interested with the effect of the geometry and/or the location of the curve $\G$ on  the existence of minimmizer for $\mu_{h}(\O,\G)$.\\
We not that for $\s=0$,  \eqref{eq:min-to-study} reduces to the famous Brezis-Nirenberg problem \cite{BN}. In this case,  for $N\geq 4$ it is enough that $h(y_0)<0$ to get a minimizer, whereas for $N=3$, the problem is no more local and existence of minimizers is guaranteed by the positiveness of  a certain mass  ---the trace of the regular part of the  Green function of the operator $-\D+h$ with zero Dirichlet data, see Druet  \cite{Druet}.   For $\s=2$, the problem reduces to 	 a linear eigenvalue problem with Hardy potential,  existence and nonexistence results  were  obtained by the second author in \cite{Thiam}.  \\
Here, we deal with the case $\s\in (0,2)$.  Our results  exhibit similar local/global phenomenon as in \cite{BN} and \cite{Druet}, with the additional property that for $N\geq 4$, the curvature of the curve at a point $y_0$  tells how much     $h(y_0)$ should be negative, while positive mass at a point $y_0\in \G$ is enough in dimension $N=3$. \\

Our first main result is the following
\begin{theorem}\label{th:main1}
Let $N \geq 4$, $\s\in (0,2)$ and   $\O$  be a   bounded domain of $\R^N$. Consider  $\G$ a smooth closed curve contained in $\O$.
Let  $h$ be  a continuous function such that the linear operator $-\D+h$ is coercive. Then there exists a positive constant $C_{N,\s}$, only depending on $N$ and $\s$ with the property that if  there exists  $y_0\in \G$  such that 
\begin{equation}\label{eq:h-bound-main-th-1}
h(y_0)<- C_{N,\s} |\k (y_0)|^2
\end{equation}
then $\mu_{h}\left(\O,\G\right) < S_{N,\s},$ and $\mu_{h}\left(\O,\G\right)$ is achieved by a positive function. Here $\k:\G\to \R^N$ is the curvature vector of   $\G$.
\end{theorem}
Inequality \eqref{eq:h-bound-main-th-1} in Theorem \ref{th:main1} shows that the \textit{sign} of the  directional curvatures of $\G$ is not important but the \textit{size} of the curvature $\k$ at a point is.  

For the explicit  value of $C_{N,\s}$ appearing  in \eqref{eq:h-bound-main-th-1},   we refer the reader to Proposition  \ref{Proposition2} below. It is given by weighted integrals involving     partial derivatives of $w$,  a minimizer for $S_{N,\s}$. In the case $N=4$, we have $C_{N,\s}= \frac{3}{2}$.
 
We now give a consequence    of Theorem \ref{eq:h-bound-main-th-1}  in the case where $h\equiv \lambda$ a constant function. We denote by $\l_1(\O)>0$ the first Dirichlet eigenvalue of  $-\D$ in $\O$.  It is easy to see that $-\D+\l $ is coercive for every $\l>-\l_1(\O)$.  In our next result, we will consider a curve   $\G$ with curvature vanishing at a point. This is (trivially) the case when $\Gamma$ contains a   segment.   
\begin{corollary}\label{cor:main1}
Let $N \geq 4$, $\s\in (0,2)$ and   $\O$  be a   bounded domain of $\R^N$. Consider  $\G$ a smooth closed curve contained in $\O$.
Suppose that the curvature $\k$ of $\G$ vanishes at a point. Then for every $\l\in (-\l_1(\O),0)$, we have 
 $\mu_{\l}\left(\O,\G\right) < S_{N,\s},$ and $\mu_{\l}\left(\O,\G\right)$ is achieved by a positive function.  
\end{corollary}
We observe that if  $\G=S^1_R$  a circle of radius $R>0$ and $h\equiv \l\in \R$ then condition \eqref{eq:h-bound-main-th-1} translates into 
$$
\l<-\frac{C_{N,\s}}{R^2}.
$$
Therefore, provided   $-\l_1(\O)<-\frac{C_{N,\s}}{R^2} $, we have that $\mu_{\l}(\O,S^1_R)$ is achieved for every $\l\in (-\l_1(\O),-\frac{C_{N,\s}}{R^2})$. One is thus led to find domains for   which  $-\l_1(\O)<-\frac{C_{N,\s}}{R^2} $. A particular example is given by the annulus  $\O_\e= B_{R+\e}\setminus B_{R-\e}$, which contains $S^1_R$ for $\e>0$. It is well known from e.g. the Faber-Krahn inequality that $\l_1(\O_\e)\geq \frac{c(N)}{\e^2} $,  so that  for sufficiently small $\e$, one always has  $-\l_1(\O_\e)<-\frac{C_{N,\s}}{R^2} $.\\
  
We now turn to the $3$-dimensional case.  We let   $G(x,y)$ be  the Dirichlet Green function of the operator $-\D +h$, with zero Dirichlet data. It satisfies 
\be\label{eq:Green-expan-introduction}
\begin{cases}
-\D_x G(x,y)+h(x) G(x,y)=0&  \qquad\textrm{  for every $x\in \O\setminus\{y\}$}\\
G(x,y)=0 &  \qquad\textrm{  for every $x\in\de  \O$.}
\end{cases}
\ee
In addition, for $N=3$, there exists a continuous function   $\textbf{m}:\O\to \R$ and a positive constant $c>0$ such that  
\be \label{eq:expans-Green}
 G (x,y)=\frac{c}{  |x- y|}+ c\, \textbf{m}(y)+o(1)  \qquad \textrm{ as $x \to y.$} 
\ee
We call the   function   $\textbf{m}:\O\to \R$  the    \textit{mass} of $-\D+h$ in $\O$.   We note that $-\textbf{m}$ is occasionally called the \textit{Robin function} of $-\D+h$ in the literature. We now state our second main result.
\begin{theorem}\label{th:main2}
Let $\s\in (0,2)$ and   $\O$  be a   bounded domain of $\R^3$. Consider  $\G$ a smooth closed curve contained in $\O$.
Let  $h$ be  a continuous function such that the linear operator $-\D+h$ is coercive. If $\textbf{m}(y_0)>0$, for some $y_0\in \G$, then 
 $\mu_{h}\left(\O,\G\right) < S_{3,\s},$ and $\mu_{h}\left(\O,\G\right)$ is achieved by a positive function.  
\end{theorem}
Since the mass $\textbf{m}$ is independent on the curve, Theorem \ref{th:main2} shows that the \textit{location} of the curve in the domain $\O$ --- so to intersect the positive part of $\textbf{m}$--- matters for the existence of solution in general. We note that there are situations in which  the mass is a everywhere positive.  This is the case four operator  $-\D+\l$, provided $\l\in \left(-\l_1(B_1),-\frac{1}{4}\l_1(B_1) \right) $, as  observed in Brezis-Nirenberg \cite{BN}. We therefore have the 
\begin{corollary}\label{cor:main2}
Let       $B_1$  the unit ball   of $\R^3$ and let $\G$ be any smooth closed curve contained in $B_1$. If  
$\l\in \left(-\l_1(B_1),-\frac{1}{4}\l_1(B_1) \right)$
then  $\mu_{\l}\left(\O,\G\right) < S_{3,\s},$ and $\mu_{\l}\left(\O,\G\right)$ is achieved by a positive function.  
\end{corollary}

The effect    of curvatures in the study of   Hardy-Sobolev  inequalities have been intensively studied in the recent years.  For each of these works, the sign of the curvatures at the point of singularity plays important roles for the existence   a solution.  The first paper, to our knowledge, being the one  of Ghoussoub and Kang \cite{GK} who considered the Hardy-Sobolev inequality with singularity at the boundary.  For more  results in this direction,   see the works of    Ghoussoub and Robert in \cite{GR2,GR3,GR4, GR5}, Demyanov and Nazarov \cite{DN}, Chern and Lin \cite{CL}, Lin and Li \cite{LL}, the authors and Minlend in \cite{FMT} and the references there in. 
We point out that in the pure Hardy-Sobolev case, $\s\in (0,2)$,  with singularity at the boundary, one has existence of minimizers for every dimension $N\geq 3$ as long as the mean curvature of the boundary is negative at the point singularity, see \cite{GR51}.

The Hardy-Sobolev inequality with interior singularity on Riemannian manifolds have been studied by Jaber \cite{Jaber1} and Thiam \cite{Thiam1}. Here also the impact of the scalar curvature at the point singularity plays an important role for the existence of minimizers in higher dimensions $N\geq 4$.   The paper  \cite{Jaber1} contains also existence result under positive mass condition for $N=3$. \\

We expect that  the arguments in this paper can be generalized for the case $\G\subset \O$ a $k$-dimensional closed submanifold, with $2\leq k\leq N-2$. Here we believe that the norm of the second fundamental from of $\G$ will play a crucial role for the existence of minimizers. An other problem of interest would be the case  $\G\subset \de \O$ is a $k$-dimensional submanifold of $\de\O$ with, $1\leq k\leq N-1$. In this situation, we suspect that the sign if the mean curvature of $\de\O$ at a point might influence on the the existence of minimizers. Finally we    note that  Ghoussoub and Robert  in \cite{GR4} obtained several results for  the case $\G$ a subspace of dimension $k\geq 2$, and among other results, if $\G$ intersects $\de\O$ transversely, they  obtain existence results under some negativity assumptions on the mean curvature.

The proof of Theorem \ref{th:main1} and Theorem \ref{th:main2} rely on test function methods. Namely to build appropriate test functions allowing to compare $\mu_h(\O,\G)$ and $ S_{N,\s}$.  While it  always holds that $\mu_h(\O,\G)\leq S_{N,\s}$, our main task is  to find a function for which  $\mu_h(\O,\G)<S_{N,\s}$. This then allows to recover compactness and thus every minimizing sequence  for $ \mu_h(\O,\G)$ converges to   a minimizer. Building these approximates solutions requires to have sharp decay estimates of a  minimizer $w$ for $S_{N,\s}$, see Section \ref{s:Premlim}.   In Section \ref{s:Exits-N-geq4}, we treat the case $N=4$ in the spirit of Aubin \cite{Aubin}. Here  we find a  continuous familly  of test functions $(u_\e)_{\e>0}$ concentrating at a point $y_0\in \G$ which yields $\mu_h(\O,\G)<S_{N,\s}  $, as $\e\to 0$, provided \eqref{eq:h-bound-main-th-1} holds.   In Section \ref{s:3D-case}, we consider   the case $N=3$, which is  more difficult. Here we use the argument of Schoen \cite{Schoen} to build our test function.
 However we cannot adopt the method in \cite{Schoen} straightforwardly. In fact,  in contrast to the case $N\geq 4$,  we could only find a descrete family of test function $(\Psi_{\e_n})_{n\in \N}$ that leads to the inequality $\mu_h(\O,\G)<S_{3,\s}$. This is due to the fact that   the (flat) ground-state $w$ for $S_{3,\s}$, $\s\in (0,2)$, is not known explicitly,   it is not radially symmetric, it is not smooth, and $S_{3,\s}$ is only invariant under translations in the $t-$direction. As in \cite{Schoen}, we use some global test functions. These are
similar to the  test  functions $(u_{\e_n})_{n\in \N}$ in dimension $N\geq 4$ near the concentration point $y_0$, but away from it is  substituted  with the regular part of the  Green function $G(x,y_0)$, which makes appear the mass $\textbf{m}(y_0)$ in its first order Taylor expansion, see \eqref{eq:expans-Green}.    \\

\bigskip
\noindent
\textbf{Acknowledgement:}
This work is  supported by the Alexander von Humboldt Foundation and the German Academic Exchange Service (DAAD). Part of the paper was written while the authors visited the Institute of Mathematics of the Goethe-University Frankfurt. They wish to thank the institute for its hospitality and the  DAAD for funding  the visit of  E.H.A.T. within the program 57060778. M.M.F. is    partially supported by the ERC Advanced Grant 2013
n. 339958 “Complex Patterns for Strongly Interacting Dynamical Systems - COMPAT” and  E.H.A.T is partially supported by the  AIMS-NEI Small Research Grant.




%
\section{Geometric Preliminaries}\label{s:Geometric-prem}
Let    $\G\subset \R^N$ be  a smooth closed  curve. Let $(E_1;\dots; E_N)$ be an orthonormal basis of $\R^N$.
 For $y_0\in \G$ and $r>0$ small,  we consider the curve $\gamma:\left(-r, r\right) \to \G$,   parameterized by arclength such that $\gamma(0)=y_0$. Up to a translation and a rotation,  we may assume that $\g'(0)=E_1$.     We choose a smooth   orthonormal frame field $\left(E_2(t);...;E_N(t)\right)$ on the normal bundle of $\G$ such that $\left(\g'(t);E_2(t);...;E_N(t)\right) $ is an oriented basis of $\R^N$ for every $t\in (-r,r)$, with $E_i(0)=E_i$. \\
We fix the following notation, that will be used  a lot in the paper,
$$
 Q_r:=(-r,r)\times B_{\R^{N-1}}(0,r) ,
$$
where $B_{\R^k}(0,r)$ denotes the ball in $\R^k$ with radius $r$ centered at the origin.
 Provided $r>0$  small, the map $F_{y_0}: Q_r\to \O$, given by 
$$
 (t,z)\mapsto  F_{y_0}(t,z):= \gamma(t)+\sum_{i=2}^N z_i E_i(t),
$$
is smooth and parameterizes a neighborhood of $y_0=F_{y_0}(0,0)$. We consider $\rho_\G:\G\to \R$ the distance function to the curve given by 
$$
\rho_\G(y)=\min_{\ov y\in \R^N}|y-\ov y|.
$$
In the above coordinates, we have 
\begin{equation}\label{eq:rho_Gamm-is-mod-z}
\rho_\G\left(F_{y_0}(x)\right)=|z| \qquad\textrm{ for every $x=(t,z)\in Q_r.$} 
\end{equation}
Clearly, for every $t\in (-r,r)$ and $i=2,\dots N$, there  are real numbers $\k_i(t)$ and $\tau^j_i(t)$  such that
\begin{equation}\label{DerivE_i}
E_i^\prime(t)= \kappa_i(t) \gamma^\prime(t)+\sum_{ {j=2} }^N\tau_i^j(t) E_j(t).
\end{equation}
The quantity  $\kappa_i(t)$ is the curvature in the $E_i(t)$-direction while $\tau_i^j(t)$ is the torsion from the osculating plane spanned by $\{\g'(t); E_j(t)\}$ in the direction  $E_j $.  We note that provided $r>0$ small, $\k_i$ and $\t_i^j$ are smooth functions on $(-r,r)$. Moreover, it is easy to see that 
\be\label{eq:tau-antisymm}
\t^j_i(t)=-\t^i_j(t) \qquad\textrm{ for $i,j=2,\dots, N$.   } 
\ee
Next, we derive the expansion of the metric induced by the parameterization $F_{y_0}$ defined above.
For $x=(t,z) \in Q_r$, we define 
$$
g_{11}(x)=   {\de_t F_{y_0}}(x) \cdot   {\de_t F_{y_0}} (x) , \qquad g_{1i}(x)=   {\de_t F_{y_0}} (x)\cdot   {\de_{z_i} F_{y_0}}(x) ,\qquad  g_{ij}(x)=   {\de_{z_j} F_{y_0}}(x) \cdot   {\de_{z_i} F_{y_0}}(x).
$$
We have the following result.
\begin{lemma}\label{MaMetric}
There exits $r>0$, only depending on $\G$ and $N$, such that  for ever $x=(t,z)\in Q_r$ 
\begin{equation}\label{Vert}
\begin{cases}
\displaystyle g_{11}(x)=1+2\sum_{i=2}^Nz_i \kappa_i(0)+2t\sum_{i=2}^Nz_i \kappa^\prime_i(0)+\sum_{ij=2}^N z_i z_j \kappa_i(0) \kappa_j(0)+\sum_{ij=2}^N z_i z_j  \b_{ij}(0)+O\left(|x|^3\right)\\
\displaystyle g_{1i}(x)=\sum_{j=2}^N z_j \tau^i_j(0)+t\sum_{j=2}^N z_j \left(\tau^i_j\right)^\prime(0)+O\left(|x|^3\right)\\
\displaystyle g_{ij}(x)=\d_{ij},
\end{cases}
\end{equation}
where  
$$
\b_{ij}(t):=\sum_{l=2}^N \tau_i^l(t) \tau_j^l(t).
$$
\end{lemma}
\begin{proof}
To alleviate the notations, we will write $F=F_{y_0}$.
We have 
\begin{equation}\label{Derivatives}
 {\de_t F}(x) =\gamma^\prime(t)+\sum_{j=2}^N z_j E_j^\prime(t)
\qquad \textrm{and}\qquad
 {\de_{z_i} F}(x) = E_i(t).
\end{equation}
Therefore
\begin{equation}\label{gij}
g_{ij}(x)=E_i(t)\cdot E_j(t)=\d_{ij}.
\end{equation}
By \eqref{DerivE_i} and \eqref{Derivatives}, we have
\begin{equation}\label{g_i1}
g_{1i}(x) = \sum_{l=2}^N z_l E_l^\prime(t)\cdot E_i(t)  =\sum_{j=2}^N z_j \tau_j^i(t)
\end{equation}
and
\begin{equation}\label{g_11}
g_{11}(x)= {\de_t F}(x) \cdot {\de_t F}(x) =1+2\sum_{i=2}^N z_i \kappa_i(t)+\sum_{ij=2}^N  z_i z_j \kappa_i(t)\kappa_j(t)
+\sum_{ij=2}^N z_i z_j \left(\sum_{l=2}^N \tau_i^l(t) \tau_j^l(t)\right).
\end{equation}
By   Taylor expansions, we get
$$
\kappa_i(t)=\kappa_i(0)+t\kappa_i^\prime(0)+O\left(t^2\right) 
\qquad \textrm{and}\qquad
\tau_i^k(t)=\tau_i^k(0)+t \left(\tau_i^k\right)^\prime(0)+O\left(t^2\right).
$$
Using these identities in      \eqref{g_11} and \eqref{g_i1},   we get \eqref{Vert}, thanks to   \eqref{gij}. This ends the proof of the lemma.
\end{proof}
As a consequence we have the following result.
\begin{lemma}\label{MaMetricMetric}
There exists $r>0$ only depending on $\G$ and $N$, such that for every $x\in Q_r$, we have 
\begin{equation}\label{DeterminantMetric}
\sqrt{|g|}(x)=1+\sum_{i=2}^N z_i \kappa_i(0)+t\sum_{i=2}^N z_i \kappa_i^\prime(0)+\frac{1}{2}\sum_{ij=2}^N z_i z_j \kappa_i(0) \kappa_j(0)+O\left(|x|^3\right),
\end{equation}
where  $|g|$ stands for the determinant of $g$.
Moreover  $g^{-1}(x)$, the matrix  inverse   of $g(x)$,   has components given by
\begin{equation}\label{InverseMetric}
\begin{cases}
\displaystyle g^{11}(x)=1-2\sum_{i=2}^N z_i \kappa_i(0)-2t\sum_{i=2}^N z_i \kappa_i^\prime(0)+3\sum_{ij=2}^N z_i z_j \kappa_i(0) \kappa_j(0)+O\left(|x|^3\right)\\
\displaystyle g^{i1}(x)=-\sum_{j=2}^N z_j \tau^i_j(0)-t\sum_{j=2}^N z_j \left(\tau^i_j\right)^\prime(0)+2\sum_{j=2}^N z_l z_j \kappa_l(0) \tau^i_j(0)+O\left(|x|^3\right)\\
\displaystyle g^{ij}(x)=\d_{ij}+\sum_{lm=2}^N z_l z_m \tau^j_l(0) \tau^i_m(0)+O\left(|x|^3\right).
\end{cases}
\end{equation}
 
\end{lemma}
\begin{proof}
We write
$$
g(x)=id+H(x),
$$
where $id$ denotes the  identity matrix on $\R^N$  and $H $ is a symmetric matrix with components $H_{\a\b}$, for $\a,\b=1,\dots,N$, given  by
\be \label{eq:g-eq-id-H}
\begin{cases}
\displaystyle H_{11}(x)=2\sum_{i=2}^N z_i \kappa_i(0)+2t\sum_{i=2}^N z_i \kappa^\prime_i(0) +\sum_{ij=2}^N z_i z_j \kappa_i(0) \kappa_j(0)+\sum_{ij=2}^N z_i z_j \b_{ij}(0)+O\left(|x|^3\right)\\
\displaystyle H_{1i}(x)= \sum_{j=2}^N z_i \tau^i_j(0)+O\left(|x|^2\right)\\
H_{ij}(x)=0.
\end{cases}
\ee
We recall that as $|H| \to0$,
\begin{equation}\label{Exp}
\sqrt{|g|}=\sqrt{\det\left(I+H\right)}=1+\frac{\tr H}{2}+\frac{\left(\tr H\right)^2}{4}-\frac{\tr (H^2)}{4}+O\left(|H|^3\right).
\end{equation}
Now  by \eqref{eq:g-eq-id-H},  as $|x|\to 0$, we have
\be \label{eq:trH-ov-2}
\frac{\tr H}{2}=\sum_{i=2}^Nz_i \kappa_i(0)+t\sum_{i=2}^Nz_i \kappa^\prime_i(0)+\frac{1}{2}\sum_{ij=2}^N z_i z_j \kappa_i(0) \kappa_j(0)+\frac{1}{2}\sum_{ij=2}^N z_i z_j  \b_{ij}(0)+O\left(|x|^3\right),
\ee
so that  
\begin{equation}\label{l2}
\frac{\left(\tr H\right)^2}{4}=\sum_{ij=2}^N z_i z_j \kappa_i(0) \kappa_j(0)+O\left(|x|^3\right).
\end{equation}
Moreover, from \eqref{eq:g-eq-id-H}, we deduce that
$$
\tr (H^2)(x)=\sum_{\a=1}^N \left(H^2(x)\right)_{\a\a}=\sum_{\a\b=1}^N H_{\a \b}(x) H_{\b\a}(x) =\sum_{\a\b=1}^N  H^2_{\a \b}(x) =H^2_{11}(x)+2\sum_{i=2}^N H^2_{i1}(x),
$$
so that
\begin{equation}\label{l3}
-\frac{\tr (H^2)}{4}=-\sum_{ij=2}^N z_i z_j \kappa_i(0) \kappa_j(0)-\frac{1}{2} \sum_{ijl=2}^N z_i z_j \tau^l_i(0) \tau^l_j(0)+O\left(|x|^3\right).
\end{equation}
Therefore plugging the expression from \eqref{eq:trH-ov-2},    \eqref{l2} and \eqref{l3} in \eqref{Exp},   we get
$$
\sqrt{|g|}(x)=1+\sum_{i=2}^N z_i \kappa_i(0)+t\sum_{i=2}^N z_i \kappa_i^\prime(0)+\frac{1}{2}\sum_{ij=2}^N z_i z_j \kappa_i(0) \kappa_j(0)+O\left(|x|^3\right).
$$
The proof of \eqref{DeterminantMetric} is thus finished.\\

By Lemma \ref{MaMetric} we can write 
$$
g(x)=id+A(x)+B(x)+O\left(|x|^3\right),
$$
where   $A$ and $B$ are symmetric matrix with  components $(A_{\a\b})$ and  $(A_{\a\b})$, $\a,\b=1,\dots,N$, given respectively by  
\begin{equation}\label{eq:HerA}
\displaystyle A_{11}(x)=2\sum_{i=2}^N z_i \kappa_i(0),\qquad A_{i1}(x)=\sum_{j=2}^N z_j \tau^i_j(0) \qquad \textrm{and} \qquad A_{ij}(x)=0 
\end{equation}
and 
\begin{equation}\label{eq:HerB}
\begin{cases}
\displaystyle  B_{11}(x)=2t\sum_{i=2}^N z_i \kappa^\prime(0)+\sum_{i=2}^N z_i z_j\kappa_i(0)\kappa_j(0)+\sum_{ij=2}^N z_i z_j \b_{ij}(0)\\
\displaystyle B_{i1}(x)=t\sum_{j=2} z_j \left(\tau^i_j\right)^\prime(0) \qquad\textrm{and} \qquad B_{ij}(x)=0.
\end{cases}
\end{equation}
We observe that, as $|x|\to 0$,  
$$
g^{-1}(x)=id-A(x)-B(x)+A^2(x)+O\left(|x|^3\right).
$$
%
We then deduce from \eqref{eq:HerA} and \eqref{eq:HerB} that 
\begin{align*}
g^{11}(x)
&\displaystyle=1-A_{11}(x)-B_{11}(x)+A_{11}^2(x)+\sum_{i=1}^N A_{1i}^2(x)+O\left(|x|^3\right) \nonumber\\
&\displaystyle=1-2\sum_{i=2}^N z_i \kappa_i(0)-2t\sum_{i=2}^N z_i \kappa^\prime(0)+3\sum_{i=2}^N z_i z_j\kappa_i(0)\kappa_j(0)+3\sum_{ij=2}^N z_i z_j \b_{ij}(0)+O\left(|x|^3\right),
\end{align*}

\begin{align*}
g^{i1}(x)&\displaystyle=-A_{1i}(x)-B_{1i}(x)+\sum_{\a=1}^N A_{i\a} A_{1\a}+O\left(|x|^3\right)\hspace{8cm} \nonumber\\
&\displaystyle=-A_{1i}(x)-B_{1i}(x)+A_{i1}(x) A_{11}(x)+\sum_{j=2}^N A_{ij}(x) A_{1j}(x)+O\left(|x|^3\right) \nonumber\\
&\displaystyle=-\sum_{j=2}^N z_j \tau^i_j(0)-t\sum_{j=2} z_j \left(\tau^i_j\right)^\prime(0)+2\sum_{jl=2}^N z_l z_j \kappa_l(0) \tau^i_j(0)
\end{align*}

and 
\begin{align*}
g^{ij}(x)&\displaystyle=\d_{ij}-A_{ij}(x)-B_{ij}(x)+\left(A^2\right)_{ij}(x)+O\left(|x|^3\right) \hspace{8cm}  \nonumber\\
&\displaystyle=\d_{ij}-A_{ij}(x)-B_{ij}(x)+A_{1i} A_{1j}+\sum_{l=2}^N A_{il}(x) A_{jl}(x)+O\left(|x|^3\right) \nonumber\\
&\displaystyle =\d_{ij}+\sum_{{lm=2} }^N z_l z_m \tau^i_m(0) \tau^j_l(0)+O\left(|x|^3\right).
\end{align*}
This ends the proof.
\end{proof}
\section{Some preliminary results}\label{s:Premlim}
We consider  the  best constant for the cylindrical Hardy-Sobolev inequality   
$$
S_{N,\s}= \min \left\{ \int_{\R^N} |\n w|^2 dx\,:\,w\in \cD^{1,2}(\R^N),\,    \int_{\R^N} |z|^{-\s} |w|^{2^*_\s} dx=1 \right\}.
$$
As mentioned in the first section,  it is   attained  by a positive function $w\in \cD^{1,2}(\R^N)$, satisfying
\begin{equation}\label{ExpoEA}
-\D w=S_{N,\s} |z|^{-\s} w^{2^*_\s-1} \qquad \textrm{ in } \R^N,
\end{equation}
see e.g. \cite{BT}. Moreover from  \cite{FMS}, we have   
\begin{equation}\label{eq:AE}
w(x)=w(t,z)=\theta\left(|t|, |z|\right) \qquad \textrm{ for a function} \qquad \theta:\R_+ \times \R_+ \to \R_+.
\end{equation}
Next we prove further decay properties of $w$ involving its higher derivatives.  We start with the following results.
\begin{lemma}\label{lem:dec-est-theta}
Let $\th$ be given by \eqref{eq:AE}. Then we have    the following properties.
\begin{itemize}
\item[(i)] The function $t\mapsto \th(t,\rho)$ is of class $C^{\infty}$ with all its derivatives uniformly bounded with respect to  $\rho$. 
\item[(ii)] There exists a constant $C>0$ such that for $|(t,\rho)|\leq 1$, we have 
$$
\th_\rho(t,\rho)+ \th_{t\rho}(t,\rho) +\rho \th_{\rho \rho}(t,\rho) \leq C \rho^{1-\s} .
$$
\end{itemize} 
\end{lemma}
\begin{proof}
For the proof of $(i)$, see \cite{FMS}.  To prove $(ii)$, we first use polar coordinates to deduce that  
\be\label{eq:eq-satis-th}
\rho^{2-N} (\rho^{N-2} \th_\rh)_\rho+\th_{tt}= S_{N,\s} \rho^{-\s} \th^{2^*_\s-1} \qquad\textrm{ for $t,\rh\in \R_+$}.
\ee
Integrating this identity in  the $\rho$ variable, we therefore get, for every $\rho>0$,
$$
\th_\rho(t,\rho)=\frac{-1}{\rho^{N-2}}\int_0^\rho  r^{N-2}  \th_{tt} (t,r)dr+ S_{N,\s}  \frac{1}{\rho^{N-2}}\int_0^\rho r^{N-2} r^{-\s} \th^{2^*_\s-1} (t,r) dr.
$$
Moreover, we have 
$$
\th_{t\rho}(t,\rho)=\frac{-1}{\rho^{N-2}}\int_0^\rho  r^{N-2}  \th_{ttt} (t,r)dr+ S_{N,\s}  \frac{1}{\rho^{N-2}}\int_0^\rho r^{N-2} r^{-\s} \de_t\th(t,r) \th^{2^*_\s-2} (t,r) dr.
$$
By  $(i)$ and the fact that  $2^*_\s\geq 2$, we obtain
$$
|\th_\rho(t,\rho)|+ |\th_{t\rho}(t,\rho)| \leq  C \rho+C\rho^{1-\s}  \leq C \rho^{1-\s} \qquad\textrm{ for $|(t,\rho)|\leq 1$}.
$$ 
Now using this in \eqref{eq:eq-satis-th}, we get $| \th_{\rh\rh}|\leq C \rh^{-\s}$, for $|(t,\rho)|\leq 1$.   The proof of $(ii)$ is completed.
 \end{proof}
As a consequence we derive decay estimates of the derivatives of $w$ up to order two.
\begin{corollary}\label{cor:dec-est-w}
Let $w$ be a ground state for $S_{N,\s}$ then there exist positive   constants $C_1,C_2$, only depending on $N$ and $\s$, such that 
\item[(i)] For every $x\in \R^N$  
  \begin{equation}\label{DecayEstimates111}
\frac{C_1}{1+|x|^{N-2}}\leq w(x) \leq \frac{C_2}{1+|x|^{N-2}} .
\end{equation}

\item[(ii)] For   $|x|= |(t,z)|\leq 1$  
$$ 
|\n w (x)|+ |x| |D^2 w (x)|\leq C_2 |z|^{1-\s}
$$
\item[(iii)] For   $|x|= |(t,z)|\geq 1$ 
$$
|\n w(x)|+ |x| |D^2 w(x)|\leq C_2 \max(1, |z|^{-\s})|x|^{1 -N}.
$$
\end{corollary}
\begin{proof}
For the proof of $(i)$, we refer to \cite[Lemma 3.1]{FMS}. The proof of $(ii)$ is an immediate consequence of Lemma \ref{lem:dec-est-theta}(ii), recalling that $w(t,z)=\th(|t|,|z|)$. Now $(iii)$ follows by Kelvin transform, using that the function $v: \R^N\to  \R$, given by $v(t,z)=v(x)=\th (|t| |x|^{-2},|z||x|^{-2} ) |x|^{2-N}$ is also a  ground state for $S_{N,\s}$, thus it satisfies $(ii)$.
\end{proof}

%
%
%
%
%
%
%
%
%
We close this section with the following result.
\begin{lemma}\label{lem:to-prove}
Let $v\in \cD^{1,2}(\R^N)$, $N\geq 3,$ satisfy $v(t,z)=\ov\th(|t|,|z|)$, for some some function $\ov\th:\R_+\times \R_+\to \R$. Then  for $0<r<R$, we have 
\begin{align*}
 \int_{\B_{R}\setminus \B_{r}}|\n v|^2_{g}\sqrt{|g|} dx&= \int_{\B_{R}\setminus \B_{r}}|\n v|^2  dx+  \frac{|\kappa(x_0)|^2}{N-1} \int_{\B_{R}\setminus \B_{r}} |z|^2 \left|{\de_t v} \right|^2 dx\\
&+ \frac{|\kappa(x_0)|^2}{2(N-1)} \int_{\B_{R}\setminus  \B_{r}} |z|^2 |\n v|^2 dx
+O\left( \int_{\B_{R}\setminus \B_{r}} |x|^3 |\n v|^2 dx \right).
\end{align*}
\end{lemma}
\begin{proof}
It is easy to see that
\begin{align}
\int_{\B_{R}\setminus \B_{r}}|\n v|^2_{g}\sqrt{|g|} dx&= \int_{\B_{R}\setminus \B_{r}}|\n v|^2  dx+  \int_{\B_{R}\setminus \B_{r}} (|\n v|^2_{g} -|\n v|^2)\sqrt{|g|} dx \nonumber\\
&+ \int_{\B_{R}\setminus \B_{r}}|\n v|^2 (\sqrt{|g|} -1)dx. \label{eq:nv-g-eps-sqrt}
\end{align}
We recall that
\begin{align*}
|\n v|^2_{g}(x) -|\n v|^2(x)=  \sum_{\a\b=1}^N\left[ g^{\a\b}( x)- \d_{\a\b} \right] \de_{z_\a} v(x) \de_{z_\b} v(x).
\end{align*}
It then follows that 
\begin{align}\label{eq:decom-nabla-v}
 \int_{\B_{R}\setminus \B_{r}}  \left[ |\n v|^2_{g} -|\n v|^2 \right]\sqrt{|g|}dx =& \sum_{ij=2}^N \int_{\B_{R}\setminus \B_{r}} \left[g^{ij}-\d_{ij}\right] \de_{z_i} v   \de_{z_j} v   \sqrt{|g|} dx  \nonumber\\
 &+\sum_{i=2}^N\int_{\B_{R}\setminus \B_{r}} g^{i1}\left(\de_t v   \de_{z_i} v \right) \sqrt{|g|}  dx  \\
 &+ \sum_{i=2}^N\int_{\B_{R}\setminus \B_{r}} [g^{11} -1]\left(\de_t v  \right)^2\sqrt{|g|} dx  \nonumber.
\end{align}
We first use  Lemma \ref{MaMetricMetric} and \eqref{eq:tau-antisymm}, to get
\begin{align}\label{eq:A}
 \displaystyle \sum_{ij=2}^N \int_{\B_{R} \setminus \B_{r}   } &\left[g^{ij} -\d_{ij}\right] \de_{z_i} v  \de_{z_j} v   \sqrt{|g|} \,dx\nonumber\\
& = \sum_{ij=2}^N \sum_{lm=2}^N \tau^i_m(0) \tau^j_l(0) \int_{\B_{R}  \setminus \B_{r}  } z_i z_j z_l z_m\frac{|\n_z v|^2}{|z|^2}    \,dx+O\left(\int_{\B_{R}\setminus \B_{r}} |x|^3 |\n_z v|^2 \,dx\right)\nonumber\\
&=O\left( \int_{\B_{R }\setminus \B_{r}} |x|^3 |\n_z w|^2 dx\right).
\end{align}
Next, we observe that
$$
\sum_{i=2}^N\int_{\B_{R}\setminus \B_{r}} g^{i1} \left(\de_t v \cdot \de_i v \right) \sqrt{|g|}  \,dx
=\sum_{i=2}^N \int_{\B_{R}\setminus \B_{r}} \Upsilon(|t|, |z|) t  z_i g^{i1} \, dx,
$$
where $\Upsilon(|t|, |z|)=\ov \th_t(|t|,|z|) \ov  \th_\rho(|t|,|z|) \frac{1}{|t|} \frac{1}{|z|}.$
In addition, from \eqref{eq:tau-antisymm}, we see that 
$$
\sum_{ij=2}^N  \t^i_j(0)z_i z_j=\sum_{ij=2}^N (\t^i_i)'(0)z_i z_j=0.
$$
 Consequently,    from \eqref{DeterminantMetric} and \eqref{InverseMetric}, we  obtain
{\begin{align}\label{eq:gi1vt-vi}
&\displaystyle\sum_{i=2}^N\int_{\B_{R}\setminus \B_{r}} g^{i1}   {\de_t v}   \de_{z_i} v   \sqrt{|g|}  \,dx
\displaystyle=\int_{\B_{R}\setminus \B_{r}} \Upsilon(|t|, |z|) t\sum_{i=2}^N  z_i g^{i1}  \sqrt{|g|}  \, dtdz \nonumber\\
&\displaystyle=-\sum_{ij=2}^N   \tau^i_j(0) \int_{\B_{R}\setminus \B_{r}} \Upsilon(|t|, |z|) t z_i z_j \,dtdz
-   \sum_{ij=2}^N \left(\tau^i_j\right)^\prime(0)\int_{\B_{R}\setminus \B_{r}} \Upsilon(|t|, |z|) t^2  z_i z_j \,dtdz \nonumber\\
&\displaystyle\,\, +2 \sum_{ijl=2}^N  \kappa_l(0) \tau_i^j(0)  \int_{\B_{R}\setminus \B_{r}} \Upsilon(|t|, |z|)t z_i z_j z_l\, dtdz    - \sum_{ijl=2}^N   \kappa^\prime(0) \tau^i_j(0)\int_{\B_{R}\setminus \B_{r}} \Upsilon(|t|, |z|) z_l z_i z_j t^2\, dtdz \nonumber \\
&\displaystyle \,\,- \sum_{ijl=2}^N   \tau^i_j(0) \kappa_l(0) \int_{\B_{R}\setminus \B_{r}} \Upsilon(|t|, |z|) z_l z_i z_j t \,dtdz+O\left(\int_{\B_{R}\setminus \B_{r}} |x|^3 |\n v|^2 \,dx\right) \nonumber\\
&\displaystyle = O\left( \int_{\B_{R}\setminus \B_{r}} |x|^3 |\n v|^2 \,dx\right).
\end{align}
}
By \eqref{DeterminantMetric} and \eqref{InverseMetric},  we have
$$
\int_{\B_{R}\setminus \B_{r}} \left| {\de_t v} \right|^2 \left[g^{11} -1\right]  \sqrt{|g|}  \,dx
=
  \frac{|\kappa(x_0)|^2}{N-1} \int_{\B_{R}\setminus \B_{r}} |z|^2 \left| {\de_t v}\right|^2 dx+O\left( \int_{\B_{R}\setminus \B_{r}} |x|^3 \left| {\de_t v} \right|^2 dx\right).
$$
Using this, \eqref{eq:A} and  \eqref{eq:gi1vt-vi}  in \eqref{eq:decom-nabla-v}, we then deduce that 
\be \label{eq:nv-eps-minus-nv}
\begin{array}{ll}
&\displaystyle  \int_{\B_{R}\setminus \B_{r}}  \left[ |\n v|^2_{g }  -|\n v|^2  \right]   \sqrt{|g|}  \,dx= \frac{|\kappa(x_0)|^2}{N-1} \int_{\B_{R}\setminus \B_{r}} |z|^2 \left| {\de_t v} \right|^2 dx+O\left(  \int_{\B_{R}\setminus \B_{r}} |x|^3 |\n v|^2 dx\right).
\end{array}
\ee
 Now by \eqref{DeterminantMetric} and \eqref{InverseMetric}, we also have that 
$$
  \int_{\B_{R }\setminus \B_{r }}|\n v|^2 (\sqrt{|g |} -1)dx= 
 \frac{|\kappa(y_0)|^2}{2(N-1)} \int_{\B_{R }\setminus  \B_{r }} |z|^2 |\n v|^2 dx + O\left(  \int_{\B_{ R}} |x|^3 |\n v|^2 dx\right).
$$
This with \eqref{eq:nv-eps-minus-nv} and  \eqref{eq:nv-g-eps-sqrt} give the desired result.

\end{proof}
%
%
%
%
%


%
\section{Existence of minimzers for $\mu_h(\O,\G)$ in dimension $N\geq 4$}\label{s:Exits-N-geq4}
We consider $\O$ a bounded domain of $\R^N$, $N\geq3$, and $\G\subset \O$ be a smooth closed curve.  For $u\in H^1_0(\O)\setminus\{0\}$, we define the ratio
\be\label{eq:def-J-u} 
J\left(u\right):=\frac{\displaystyle\int_\O |\n u |^2 dy+\int_\O h u^2 dy}{\left(\displaystyle\int_\O \rho^{-\s}_\G|u|^{2^*_\s} dy\right)^{2/2^*_\s}}.
\ee
We will construct a family of test function $(u_\e)_\e\in H^1_0(\O)$ and provide an expansion for $J(u_\e)$ as  $\e\to 0$.
We let 
$\eta \in \calC^\infty_c\left(F_{y_0}\left({Q}_{2r}\right)\right)$ be such that
$$
0\leq \eta \leq 1 \qquad \textrm{ and }\qquad \eta \equiv 1 \quad \textrm{in }  \B_r .
$$
For $\e>0$, we consider  $u_\e: \O \to  \R$ given  by
\begin{equation}\label{eq:TestFunction-w}
u_\e(y):=\e^{\frac{2-N}{2}} \eta(F^{-1}_{y_0}(y)) w \left(\frac{F^{-1}_{y_0}(y)}{\e} \right).
\end{equation}
In particular,  for every $x=(t,z)\in \R\times \R^{N-1}$, we have 
\begin{equation}\label{eq:TestFunction-th}
u_\e\left(F_{y_0}(x)\right):=\e^{\frac{2-N}{2}}\eta\left(x\right)\th \left(\frac{|t|}{\e},\frac{|z|}{\e} \right).
\end{equation}
It is clear that $u_\e \in H^1_0(\O).$ We have the following
\begin{lemma}\label{Expansion}
For $J$  given by \eqref{eq:def-J-u}  and $u_\e$ given by \eqref{eq:TestFunction-w}, as $\e\to0$, we have 
\begin{align}\label{eq:expans-J-u-eps}
J\left(u_\e\right)=  & S_{N,\s}+  \e^2\frac{|\kappa(x_0)|^2}{N-1} \int_{  \B_{r/\e}} |z|^2 \left|{\de_t w} \right|^2 dx+\e^2\frac{|\kappa(x_0)|^2}{2(N-1)} \int_{  \B_{r/\e}} |z|^2 |\n w|^2 dx \nonumber\\
&  -\frac{\e^2}{2^*_\s} \frac{|\kappa(y_0)|^2}{(N-1)} S_{N,\s}\int_{\B_{r/\e}} |z|^{2-\s} w^{2^*_\s} dx+  \e^2  h(y_0) \int_{\B_{r/\e}} w^2 dx\\
&+O\left( \e^2   \int_{\B_{r/\e}}|h(F_{y_0}(\e x))-h(y_0)|  w^2 dx\right)+O\left(\e^{N-2}\right). \nonumber
\end{align}
 
\end{lemma}
\begin{proof}
To simplify the notations, we will write $F$  in the place of $F_{y_0}$.
Recalling \eqref{eq:TestFunction-w}, we   write
$$
u_\e(y)=\e^{\frac{2-N}{2}} \eta(F^{-1} (y)) W_\e(y),
$$
where $W_\e (y) =w \left(\frac{F^{-1} (y)}{\e} \right)$.
Then
$$
|\n u_\e |^2 =\e^{2-N} 
\left(
\eta^2 |\n W_\e|^2+\eta^2 |\n W_\e |^2+\frac{1}{2} \n W_\e^2 \cdot \n \eta^2
\right).
$$
Integrating by parts, we have
\begin{align}\label{eq:expan-nabla-u-eps}
\displaystyle \int_\O |\n u_\e|^2 dy &\displaystyle=\e^{2-N} \int_{F\left({Q}_{2r}\right)} \eta^2 |\n W_\e|^2  dy+\e^{2-N} \int_{F\left({Q}_{2r}\right)\setminus F\left(\B_{r}\right)} W_\e^2 \left(|\n \eta|^2-\frac{1}{2} \D \eta^2 \right)  dy \nonumber\\
&\displaystyle=\e^{2-N} \int_{F\left({Q}_{2r}\right)} \eta^2 |\n W_\e|^2  dy-\e^{2-N} \int_{F\left({Q}_{2r}\right)\setminus F\left(\B_{r}\right)}W_\e^2 \eta \D \eta  dy  \nonumber\\
&\displaystyle=\e^{2-N} \int_{F\left({Q}_{2r}\right)} \eta^2 |\n W_\e|^2  dy+O\left(\e^{2-N} \int_{F\left({Q}_{2r}\right)\setminus F\left(\B_{r}\right)} W_\e^2  dy\right) .
\end{align}

By the  change of variable $y=\frac{F(x)}{\e}$ and \eqref{eq:TestFunction-th}, we can apply Lemma \ref{lem:to-prove}, to get
\begin{align*}
&\displaystyle \int_\O |\n u_\e|^2 dy\displaystyle= \int_{ {Q}_{r/\e} } |\n w|^2_{g_\e}\sqrt{|g_\e|} dx+O\left(\e^{2} \int_{ {Q}_{2r/\e} \setminus  \B_{r/\e} }w^2 dx+
 \int_{ {Q}_{2r/\e} \setminus  \B_{r/\e} }|\n w|^2 dx \right)\\
&= \int_{\R^N}|\n w|^2 dx+   \e^2\frac{|\kappa(y_0)|^2}{N-1} \int_{  \B_{r/\e}} |z|^2 \left|{\de_t w} \right|^2 dx+\e^2\frac{|\kappa(y_0)|^2}{2(N-1)} \int_{  \B_{r/\e}} |z|^2 |\n w|^2 dx\\
&+  O\left(\e^3 \int_{ \B_{r/\e}} |x|^3 |\n w|^2 dx +\e^2 \int_{\B_{2r/\e}\setminus \B_{r/\e}}  |w|^2 dx+  \int_{\R^N\setminus\B_{r/\e}}|\n w|^2 dx+  \e^2 \int_{\B_{2r/\e}\setminus \B_{r/\e}}|z|^2|\n w|^2 dx \right).
%
\end{align*}
Using Corollary \ref{cor:dec-est-w}, we find  that 
\begin{align*}
\displaystyle \int_\O |\n u_\e|^2 dy &\displaystyle
=  S_{N,\s}+  \e^2\frac{|\kappa(y_0)|^2}{N-1} \int_{  \B_{r/\e}} |z|^2 \left|{\de_t w} \right|^2 dx+\e^2\frac{|\kappa(y_0)|^2}{2(N-1)} \int_{  \B_{r/\e}} |z|^2 |\n w|^2 dx+ O\left(\e^{N-2}  \right).
\end{align*}

%
By the change of variable $y=\frac{F(x)}{\e}$, \eqref{eq:AE}, \eqref{eq:rho_Gamm-is-mod-z} and   \eqref{DeterminantMetric},  we get 
\begin{align*}
&\displaystyle\int_\O \rho^{-\s}_\G |u_\e|^{2^*_\s} dy=\int_{\B_{r/\e}} |z|^{-s} w^{2^*_\s} \sqrt{|g_\e|}dx+ O\left( \int_{\B_{2r/\e}\setminus \B_{r/\e}}  |z|^{-\s} (\eta(\e x) w)^{2^*_\s} dx \right)\\
%
%
&\displaystyle= \int_{\B_{r/\e}} |z|^{-\s} w^{2^*_\s} dx+\e^2\frac{|\kappa(y_0)|^2}{2(N-1)} \int_{\B_{r/\e}} |z|^{2-\s} w^{2^*_\s} dx\\
&\quad +O\left(\e^3 \int_{\B_{r/\e}} |x|^3 |z|^{-\s} w^{2^*_\s} dx+   \int_{\B_{2r/\e}\setminus \B_{r/\e}}  |z|^{-\s}  w^{2^*_\s} dx \right)\\
&\displaystyle= 1+\e^2\frac{|\kappa(y_0)|^2}{2(N-1)} \int_{\B_{r/\e}} |z|^{2-\s} w^{2^*_\s} dx\\
&\quad +O\left(\e^3 \int_{\B_{r/\e}} |x|^3 |z|^{-\s} w^{2^*_\s} dx + \int_{\R^N \setminus \B_{r/\e}} |z|^{-\s} w^{2^*_\s} dx+ \int_{\B_{2r/\e}\setminus \B_{r/\e}}  |z|^{-\s}  w^{2^*_\s} dx\right).
\end{align*}
Using  \eqref{DecayEstimates111}, we have
$$
\e^3 \int_{\B_{r/\e}} |x|^3 |z|^{-\s} w^{2^*_\s} dx+\int_{\R^N \setminus \B_{r/\e}} |z|^{-\s} w^{2^*_\s} dx+\int_{\B_{2r/\e}\setminus \B_{r/\e}}  |z|^{-\s}  w^{2^*_\s} dx=O\left(\e^{N-\s}\right).
$$
Hence by Taylor expanding, we get
$$
\left(\int_\O \rho^{-\s}_\G |u_\e|^{2^*_\s} dx\right)^{2/2^*_\s} = 1+
\frac{\e^2}{2^*_\s}\frac{|\kappa(y_0)|^2}{(N-1)} \int_{\B_{r/\e}} |z|^{2-\s} w^{2^*_\s} dx+O\left(\e^{N-\s}\right).
$$
Finally, by \eqref{eq:expan-nabla-u-eps}, we conclude that
\begin{align*} 
J\left(u_\e\right)=  & S_{N,\s}+  \e^2\frac{|\kappa(y_0)|^2}{N-1} \int_{  \B_{r/\e}} |z|^2 \left|{\de_t w} \right|^2 dx+\e^2\frac{|\kappa(y_0)|^2}{2(N-1)} \int_{  \B_{r/\e}} |z|^2 |\n w|^2 dx\\
&  -\frac{\e^2}{2^*_\s} \frac{|\kappa(y_0)|^2}{(N-1)} S_{N,\s}\int_{\B_{r/\e}} |z|^{2-\s} w^{2^*_\s} dx+  \e^2  h(y_0) \int_{\B_{r/\e}} w^2 dx\\
&+O\left( \e^2  \int_{\B_{r/\e}} |h(F_{y_0}(\e x)- h(y_0)| w^2 dx\right)+O\left(\e^{N-2}\right).
\end{align*}
We thus get the desired result.

\end{proof}

%
%
%
%
%
%
%
%
%
\begin{proposition}\label{Proposition2}
  For $N\geq 5$, we define
$$
A_{N,\s}:= \frac{1}{N-1} \int_{  \R^N} |z|^2 \left|{\de_t w} \right|^2 dx+ \left( \frac{1}{2}-\frac{1}{2^*_\s} \right)\frac{1}{N-1} \int_{  \R^N} |z|^2 |\n w|^2 dx
  +\frac{1}{2^*_\s}   \int_{\R^N} w^2 dx >0
$$
and 
$$
B_{N,\s}:=  \int_{\R^N} w^2 dx.
$$
Assume that, for some $y_0 \in \G$, there holds
$$
\begin{cases}
\displaystyle  h(y_0)   <-\frac{A_{N,\s}}{{B_{N,\s}}} |\kappa(y_0)|^2& \qquad \textrm{for } N \geq 5\\
\displaystyle   h(y_0)<-\frac{3}{2}|\kappa(y_0)|^2 & \qquad \textrm{for } N=4.
\end{cases}
$$
Then
$$
\mu_h\left(\O,\G \right) < S_{N,\s}.
$$
\end{proposition}
\begin{proof}
We claim that
 \begin{align}\label{eq:z-sqrt-w-2-star}
S_{N,\s} \int_{Q_{r/\e}}      |z|^{2-\s}  w^{2^*_\s} dx&= \int_{Q_{r/\e}}  |z|^2 |\n w|^2 dx-(N-1)\int_{Q_{r/\e}} w^2        dx +O(\e^{N-2}).
\end{align} 
To prove this claim, we let $\eta_\e(x)=\eta(\e x)$. 
We multiply \eqref{ExpoEA} by $|z|^2\eta_\e w$ and integrate by parts to get
\begin{align*}
\displaystyle S_{N,\s} \int_{Q_{2r/\e}}    \eta_\e  |z|^{2-\s}  w^{2^*_\s} dx\displaystyle&=\int_{Q_{2r/\e}} \n w \cdot \n \left(\eta_\e  |z|^2 w\right) dx\\
&\displaystyle= \int_{Q_{2r/\e}} \eta_\e  |z|^2 |\n w|^2 dx+\frac{1}{2} \int_{Q_{2r/\e}} \n w^2 \cdot \n \left(|z|^2\eta_\e \right) dx\\
&\displaystyle=\int_{Q_{2r/\e}} \eta_\e  |z|^2 |\n w|^2 dx-\frac{1}{2} \int_{Q_{2r/\e}} w^2 \D\left(|z|^2\eta_\e \right) dx\\
&\displaystyle=\int_{Q_{2r/\e}} \eta_\e  |z|^2 |\n w|^2 dx-(N-1)\int_{Q_{2r/\e}} w^2 \eta_\e      dx\\
&\displaystyle\quad-\frac{1}{2} \int_{Q_{2r/\e}\setminus Q_{r/\e}} w^2 (|z|^2\D\eta_\e+ 4\n \eta_\e\cdot z) dx.
\end{align*} 
We then deduce that 
\begin{align*}
S_{N,\s} \int_{Q_{r/\e}}      |z|^{2-\s}  w^{2^*_\s} dx&= \int_{Q_{r/\e}}  |z|^2 |\n w|^2 dx-(N-1)\int_{Q_{r/\e}} w^2        dx\\
&+O\left( \int_{Q_{2r/\e}\setminus Q_{r/\e}}      |z|^{2-\s}  w^{2^*_\s} dx+  \int_{Q_{2r/\e}\setminus  Q_{r/\e}}  |z|^2 |\n w|^2 dx+ \int_{Q_{2r/\e}\setminus  Q_{r/\e}}     w^2 dx     \right)\\
&+ O\left(  \e \int_{Q_{2r/\e}\setminus  Q_{r/\e}}  |z|  |\n w| dx+ \e^2 \int_{Q_{2r/\e}\setminus  Q_{r/\e}}   |z|^2   w^2 dx     \right).
\end{align*} 
Thanks to Corollary \ref{cor:dec-est-w}, we  get \eqref{eq:z-sqrt-w-2-star} as claimed.

Next, by the continuity of  $h$, for $\delta>0$, we can find  $r_\d>0$ such that
\be \label{eq:cont-h-eff}
|h(y)-h(y_0)|<\d \qquad \textrm{for ever $y\in F\left(\B_{r_\d}\right)$ }.
\ee

\noindent
\textbf{Case $N\geq 5.$}\\
\noindent
Using  \eqref{eq:z-sqrt-w-2-star} and \eqref{eq:cont-h-eff}  in   \eqref{eq:expans-J-u-eps},  we obtain, for every $r\in (0,r_\d)$
\begin{align*}
J\left(u_\e\right)=  & S_{N,\s}+  \e^2\frac{|\kappa(y_0)|^2}{N-1} \int_{  \R^N} |z|^2 \left|{\de_t w} \right|^2 dx+\e^2\left( \frac{1}{2}-\frac{1}{2^*_\s} \right)\frac{|\kappa(y_0)|^2}{N-1} \int_{  \R^N} |z|^2 |\n w|^2 dx\\
&  +\frac{\e^2}{2^*_\s} {|\kappa(y_0)|^2}  \int_{\R^N} w^2 dx+  \e^2  h(y_0) \int_{\R^N} w^2 dx+O\left( \e^2 \d^2 \int_{\R^N}   w^2 dx\right)+O\left(\e^{N-2}\right),
\end{align*}
where we have used Corollary \ref{cor:dec-est-w} to get the estimates
$$
\int_{  \R^N\setminus \B_{r/\e}} |z|^2 |\n w|^2 dx+   \int_{\R^N\setminus \B_{r/\e}} w^2 dx=O(\e).
$$
It follows that, for every $r\in (0,r_\d)$,
\begin{align*}
J\left(u_\e\right)=  S_{N,\s}+\e^2\left\{ A_{N,\s}    |\kappa(y_0)|^2+ B_{N,\s} h(y_0) \right\}+O(\d\e^2   B_{N,\s})+O\left(\e^{3}\right).
\end{align*}
Suppose now that 
$$
A_{N,\s}    |\kappa(y_0)|^2+ B_{N,\s} h(y_0) <0.
$$
We can thus    choose respectively  $\delta>0$ small and    $\e>0$ small so that  $ J(u_\e)< S_{N,\s}$. Hence we get
$$
\mu_h\left(\O ,\G\right)< S_{N,\s}.
$$
\noindent
\textbf{Case $N=4.$}\\
\noindent
From \eqref{eq:expans-J-u-eps} and \eqref{eq:cont-h-eff}, we estimate, for every $r\in (0,r_\d)$
\begin{align*}
J\left(u_\e\right)\leq   & S_{N,\s}+  \e^2\frac{3|\kappa(y_0)|^2}{2(N-1)} \int_{  \B_{r/\e}} |z|^2 |\n w|^2 dx -\frac{\e^2}{2^*_\s} \frac{|\kappa(y_0)|^2}{(N-1)} S_{N,\s}\int_{\B_{r/\e}} |z|^{2-\s} w^{2^*_\s} dx\\
&+  \e^2  h(y_0) \int_{\B_{r/\e}} w^2 dx +O\left( \e^2  \d \int_{\B_{r/\e}}  w^2 dx\right)+O\left(\e^{N-2}\right). \nonumber
\end{align*}
This with \eqref{eq:z-sqrt-w-2-star} yield
\begin{align*}
J\left(u_\e\right)    & \leq S_{N,\s}+  \e^2\frac{3|\kappa(y_0)|^2}{2 (N-1)}S_{N,\s}\int_{\B_{r/\e}} |z|^{2-\s} w^{2^*_\s} dx -\frac{\e^2}{2^*_\s} \frac{|\kappa(y_0)|^2}{(N-1)} S_{N,\s}\int_{\B_{r/\e}} |z|^{2-\s} w^{2^*_\s} dx\\
&+  \e^2 \left\{ \frac{3}{2}  |\kappa(y_0)|^2+h(y_0) \right\} \int_{\B_{r/\e}} w^2 dx+O\left( \e^2  \d \int_{\B_{r/\e}}  w^2 dx\right)+O\left(\e^{N-2}\right).
\end{align*}
Since, by \eqref{DecayEstimates111},  
$$
\int_{\B_{r/\e}} |z|^{2-\s} w^{2^*_\s} dx=O(1),
$$
we therefore have 
\begin{align*}
J\left(u_\e\right)    & \leq S_{4,\s}  + \e^2 \left\{ \frac{3 |\kappa(y_0)|^2}{2} +h(y_0) \right\} \int_{\B_{r/\e}} w^2 dx+O\left( \e^2  \d \int_{\B_{r/\e}}  w^2 dx\right)+C\e^{2} ,.
\end{align*}
for some positive constant $C$ independent on $\e$.
By \eqref{DecayEstimates111}, we have that
$$
\begin{array}{ll}
\displaystyle \int_{\B_{r/\e}} \frac{C_1^2}{1+|x|^2} dx \leq \int_{\B_{r/\e}} w^2 dx \leq \int_{\B_{r/\e}} \frac{C_2^2}{1+|x|^2} dx,
\end{array}
$$
so that 
\begin{equation}\label{EtAlors0}
\begin{array}{ll}
\displaystyle \int_{B_{\R^4}(0,r/\e)} \frac{C_1^2}{\left(1+|x|^2\right)^2} dx \leq \int_{\B_{r/\e}} w^2 dx \leq \int_{B_{\R^4}(0,2r/\e)} \frac{C_2^2}{\left(1+|x|^2\right)^2} dx.
\end{array}
\end{equation}
Using polar coordinates and  a change of variable, for $R>0$, we have
$$
\begin{array}{ll}
\displaystyle \int_{B_{\R^4}(0,R)} \frac{dx}{\left(1+|x|^2\right)^2} dx
&\displaystyle= |S^3| \int_0^{R} \frac{t^3}{\left(1+t^2\right)^2} dt\\
&\displaystyle= |S^3| \int_0^{\sqrt{R}} \frac{s}{2\left(1+s\right)^2} ds\\
&\displaystyle=\frac{|S^3|}{2}\left(\log\left(1+\sqrt{R}\right)-\frac{\sqrt{R}}{1+\sqrt{R}}\right).
\end{array}
$$
Therefore, there exist numerical constants $c,\ov c>0$ such that for every  $\e>0$ small, we have  
\begin{equation}\label{EtAlors1}
 c   |\log  \e |\leq   \int_{\B_{r/\e}} w^2 dx \leq \ov c   |\log  \e |.
\end{equation}
Now we assume that  $ \frac{3|\kappa(y_0)|^2}{2}+ h(y_0)<0$. 
Therefore by Lemma \ref{Expansion}  and \eqref{EtAlors1}, we get  
$$
J\left(u_\e\right)\leq S_{4,s}+c \left\{ \frac{3}{2}  |\kappa(y_0)|^2+h(y_0)\right\}\e^2 |\log \e|+ \ov c \d\e^2  |\log \e|+ C \e^{2} .
$$
 Then choosing $\d>0$ small and     $\e$ small, respectively, we deduce that  $\mu_h\left(\O,\G\right)\leq J(u_\e)< S_{4,\s}$.
This ends the proof of the proposition.
\end{proof}
\begin{proof}[Proof of Theorem \ref{th:main1} (completed)]
Thanks to  Proposition \ref{Proposition2},  we can apply Proposition \ref{Pro3333} to get the result with $C_{N,\s}=\frac{A_{N,\s}}{B_{N,\s}}$ for $N\geq 5$ and $C_{4,\s}=\frac{3}{2}$.

\end{proof}
\section{Existence of minimizer for $\mu_h(\O,\G)$ in dimension three}\label{s:3D-case}
We consider the function
$$
\cR:\R^3\setminus\{0\}\to \R,    \qquad x\mapsto  \cR(x)=\frac{1}{|x|}
$$
which  satisfies
\begin{equation}\label{eq:Green-R3-3D}
-\D \cR=0 \qquad \textrm{ in $\R^3\setminus\{0\}$. }  
\end{equation}
We denote by $G$ the solution to the equation 
\begin{equation}\label{eq:Green-3D}
\begin{cases}
-\D_x G(y, \cdot)+h G(y,\cdot)=0& \qquad \textrm{  in $\O\setminus \{y\}$. }  \\
G(y,\cdot )=0&   \qquad \textrm{  on $\de \O $, }
\end{cases}
\end{equation}
and satisfying
\be\label{eq:expand-Green-trace}
G(x,y)=  \cR(x-y)+O(1)\qquad\textrm{ for $x, y\in \O$ and $x\not= y$.}
\ee
We note that $G$ is proportional to the Green function of $-\D+h$ with zero Dirichlet data.\\
We let $\chi\in C^\infty_c(-2,2)$ with $\chi \equiv 1$ on $(-1,1)$ and $0\leq \chi<1$. For $r>0$,   we  consider the cylindrical symmetric cut-off function
\be\label{eq:def-cut-off-cylind} 
\eta_r(t,z)=\chi\left(\frac{|t|+|z|}{r} \right) \qquad \qquad\textrm{ for every  $(t,z)\in \R\times \R^2$}.
\ee
It is clear that 
$$
\eta_r\equiv 1\quad \textrm{ in $\B_r$},\qquad \eta_r\in H^1_0({Q}_{2r}),\qquad |\n \eta_r|\leq  \frac{C}{r} \quad\textrm{ in $\R^3$}.
$$
For $y_0\in \O$, we let  $r_0\in (0,1)$  such that   
\be\label{eq:def-r0} 
y_0+ Q_{2r_0}\subset\O. 
 \ee  
We define the function $M_{y_0}: Q_{2r_0}\to \R$ given by
\begin{equation}\label{C9}
M_{y_0}(x):=  G (y_0,x+y_0)-{\eta_r}(x)\frac{1}{|x|}     \qquad \textrm{ for every $x\in Q_{2r_0}$}.
\end{equation}
 It follows from  \eqref{eq:expand-Green-trace} that  $M_{y_0}\in  L^\infty(Q_{r_0})$. By \eqref{eq:Green-3D} and \eqref{eq:Green-R3-3D}, 
 $$
|-\D {M}_{y_0}(x)+h(x) {M}_{y_0}(x)|\leq \frac{C}{|x|}= C \cR(x) \qquad \textrm{ for every  $x\in Q_{r_0}$},
 $$
 whereas $\cR\in L^p( Q_{r_0})$ for every $p\in (1,3)$. Hence by
elliptic regularity theory, $M_{y_0}\in W^{2,p}(Q_{r_0/2})$ for every $p\in (1,3)$. Therefore by Morrey's embdding theorem, we deduce that 
\be \label{eq:regul-beta}
\|M_{y_0}\|_ {C^{1,\varrho}(Q_{r_0/2})}\leq C \qquad \textrm{ for every $\varrho\in (0,1)$.}
\ee
In view of \eqref{eq:expans-Green}, the mass of the operator $-\D+h$ in $\O$ at the point $y_0\in   \O$ is given by  
\be \label{eq:def-mass}
  \textbf{m}(y_0)={M}_{y_0}(0).
\ee
We recall that the positive  ground state solution $w$  
satisfies 
\begin{equation}\label{eq:w-gorund-3D}
-\D w=S_{3,\s}|z|^{-\s} w^{2^*_\s-1} \qquad \textrm{in } \R^3,\qquad \int_{\R^3}|z|^{-\s} w^{2^*_\s}dx=1,
\end{equation}
where $x=(t,z)\in \R \times \R^{2}$. In addition by \eqref{DecayEstimates111}, we have 
\begin{equation}\label{eq:up-low-bound-w-3D}
\frac{C_1}{1+|x|} \leq w(x) \leq \frac{C_2}{1+|x|} \qquad \textrm{ for every $x\in   \R^3$.}
\end{equation}
The following result     will be crucial in the following of    this section.
\begin{lemma}\label{lem:v-to-cR}
Consider the function $v_\e:  \R^3\setminus\{0\}\to \R$ given by 
$$
v_\e(x)= \e^{-1} w\left(\frac{x}{\e}\right).
$$
Then there exists a   constant $\textbf{c}>0$ and a sequence $(\e_n)_{n\in \N}$    (still denoted by $\e$)  such that 
$$
v_\e (x) \to  \frac{\textbf{c}}{|x|}  \qquad \textrm{ and } \qquad \n v_\e (x) \to -\textbf{c}  \frac{x}{|x|^3} \qquad \textrm{ for  all most every  $x\in \R^3 $ } 
$$
and 
\be\label{eq:nv-eps-to-nv-C1}
v_\e (x) \to  \frac{\textbf{c}}{|x|}  \qquad \textrm{ and } \qquad  \n v_\e (x) \to    -\textbf{c}  \frac{x}{|x|^3} \qquad \textrm{ for  every   $x\in \R^3\setminus\{z=0\}$.  } 
\ee

\end{lemma}
\begin{proof}
By Corollary \ref{cor:dec-est-w}, we have that  $(v_\e)$ is bounded in $C^2_{loc}(\R^3\setminus\{z=0\})$. Therefore by Arzel\'a-Ascolli's theorem $v_\e$ converges to $v$ in $C^1_{loc}(\R^3\setminus\{z=0\})$. In particular,
$$
v_\e \to v  \qquad \textrm{ and } \qquad \n v_\e \to \n v \qquad\textrm{ almost every where on $\R^3$.}
$$  
It is plain, from \eqref{eq:up-low-bound-w-3D}, that 
\be\label{eq:up-low-bound-v-eps-3D}
 0<\frac{C_1}{\e+|x|} \leq v_\e(x) \leq \frac{C_2}{\e+ |x|} \qquad \textrm{ for almost every $x\in   \R^3 $.}
\ee
By \eqref{eq:w-gorund-3D}, we have
\begin{equation}\label{eq:eq-for-v-eps}
-\D v_{\e}(x)={\e}^{2-\s} f_{\e}(x) \qquad \textrm{ in } \R^3,
\end{equation}
where
$$
f_{\e}(x)= S_{3,\s} |z|^{-\s} v_{\e}^{2^*_\s-1}(x)\leq  C |z|^{-\s} |x|^{-5+2\s}\qquad \textrm{ for almost  every $x=(t,z)\in \R^3 $. }
$$
We let $\varphi \in C^\infty_c\left(\R^3\setminus \lbrace 0\rbrace\right)$. We multiply \eqref{eq:eq-for-v-eps} by $\varphi$ and integrate by parts to get
$$
-\int_{\R^3} v_{\e} \D \varphi dx= \e^{2-\s} \int_{\R^3} f_{\e}(x) \varphi (x) dx.
$$
By \eqref{eq:up-low-bound-v-eps-3D} and the dominated convergence theorem, we can pass to the limit in the above identity and deduce that 
$$
\Delta v=0 \qquad \quad\textrm{ in } \calD^\prime\left(\R^3\setminus \lbrace 0\rbrace\right).
$$
In particular $v$ is equivalent to a function of class $C^\infty\left(\R^3 \setminus \lbrace0\rbrace\right)$ which is still denoted by $v$. Thanks to \eqref{eq:up-low-bound-v-eps-3D}, by B\^{o}cher's theorem, there exists a constant $\textbf{c}>0$  such that
$$
v(x)=\frac{\textbf{c}}{|x|}.
$$
The proof of the lemma is thus finished.
\end{proof}

%
%
%
%
We start by recording  some useful estimates.
\begin{lemma}\label{lem:estimates-for-3D}
There exists a constant $C>0$ such that for every  $\e,r\in (0,r_0/2)$, we have 
\be\label{eq:est-nw-sq-3D}
  \int_{  \B_{r/\e}}   |\n w |^2 dx\leq C\max\left(1, \frac{\e}{r}\right),\qquad    \int_{  \B_{r/\e}}   | w |^2 dx\leq C \max\left(1 , \frac{r}{\e} \right), 
\ee
\be \label{eq:est-wnw-sq-3D} 
   \int_{  \B_{r/\e}}  w  |\n w|   dx\leq C \max\left(  1, \log\frac{r}{\e}\right),
\ee
\be \label{eq:est-nw-3D}
  \int_{ \B_{r/\e}}   |\n  w| dx \leq C \max\left(1 , \frac{r}{\e} \right), \qquad   \int_{ \B_{r/\e}}   |  w| dx \leq  C \max\left(1 , \frac{r^2}{\e^2} \right)   
\ee
%
and 
\be\label{eq:est-L2-star-3D}
  \e^2 \int_{\B_{r/\e}} |z|^{-\s}|x|^2  w ^{2^*_\s}  dx+\e   \int_{\B_{4r/\e}\setminus \B_{r/\e}} |z|^{-\s}  w ^{2^*_\s-1}  dx+  \int_{\R^3\setminus \B_{r/\e}} |z|^{-\s}  w ^{2^*_\s}  dx  \leq  C r^{\s-3} \e^{3-\s}.
\ee
\end{lemma}
\begin{proof}
The proof of this lemma is not difficult and uses only the estimates in Corollary \ref{cor:dec-est-w}. We therefore skip the details. 
%
%

\end{proof}
  
\subsection{Proof of Theorem \ref{th:main2}}\label{ss:proof-of-th-3D}

Given $y_0\in \G\subset\O\subset \R^3$, we let $r_0$ as defined in \eqref{eq:def-r0}. For    $r\in (0, r_0/2)$,  we consider $F_{y_0}: Q_r\to \O$ (see Section \ref{s:Geometric-prem}) parameterizing a neighborhood of $y_0$ in $\O$, with the property that $F_{y_0}(0)=y_0$.
For   $\e>0$,  we consider  $u_\e: \O \to  \R$ given  by
$$
u_\e(y):=\e^{-1/2} \eta_r(F^{-1}_{y_0}(y)) w \left(\frac{F^{-1}_{y_0}(y)}{\e} \right).
$$
We can now define the test function $\Psi_\e:\O\to \R$ by
\be 
\label{eq:TestFunction-Om-3D}
\Psi_\e\left(y\right)=u_\e(y)+\e^{1/2}  \textbf{c}\,  \eta_{2r}(F^{-1}_{y_0}(y) ){M}_{y_0}(F^{-1}_{y_0}(y) ).
\ee
It is plain that $\Psi_\e\in H^1_0(\O)$ and 
$$
\Psi_\e\left(F_{y_0}(x)\right)=\e^{-1/2} \eta_r(x) w \left(\frac{x}{\e} \right)+\e^{1/2}  \textbf{c} \,  \eta_{2r}(x) {M}_{y_0}(x) \qquad\textrm{ for every $x\in \R^N$.}
$$
The main result of this section is contained in the following
\begin{proposition}\label{Expansion-no-h}
Let    $(\e_n)_{n\in \N}$ and $\textbf{c}$ be the sequence and the number given by Lemma \ref{lem:v-to-cR}. Then there exists  $r_0,n_0>0$ such that for every $r\in (0,r_0)$ and $n\geq n_0$
 \begin{align*}
\displaystyle  J(\Psi_\e):= \frac{\displaystyle  \int_\O |\n \Psi_{\e_n}|^2 dy  +\int_\O h | \Psi_{\e_n}|^2 dy   }{\displaystyle\left( \int_{\O} \rho^{-\s}_\G |\Psi_{\e_n}|^{2^*_\s} dy \right)^{\frac{2}{2^*_\s}}}
&=  S_{3,\s}-  \e_n \pi^2 \textbf{m}(y_0)    \textbf{c}^2+\calO_r(\e_n) ,
\end{align*}
for some numbers   $\calO_r(\e_n)$ satisfying    
$$
\lim_{r\to 0}\lim_{n\to \infty}  \e^{-1}_n \calO_r(\e_n)=0.
$$
\end{proposition}
The proof of this proposition will be separated  into two steps given by Lemma \ref{lem:expans-num-3D} and Lemma \ref{lem:expans-denom-3D} below.
To alleviate the notations, we will write $\e$ instead of $\e_n$ and  we will remove the subscript  $y_0$, by writing $M$ and $F$ in the place of ${M}_{y_0}$ and $F_{y_0}$ respectively.
We define 
$$
\ti \eta_r(y):=\eta_r(F^{-1}(y)),\qquad V_\e(y):=v_\e(F^{-1}(y)) \qquad\textrm{ and } \qquad \ti M_{2r}(y):=\eta_{2r}(F^{-1}(y) )  M ( F^{-1}(y))  ,
$$
where $v_\e(x)=\e^{-1} w\left(\frac{x}{\e} \right).$ With these notations, \eqref{eq:TestFunction-Om-3D} becomes
\be  \label{eq:def-W-eps-3D}
\Psi_\e (y) = u_\e(y)+ \e^{\frac{1}{2}}  \textbf{c}\, \ti M_{2r}(y)= \e^{\frac{1}{2}}   V_\e(y)+    \e^{\frac{1}{2}}  \textbf{c}\, \ti M_{2r}(y) .
\ee
We first  consider  the numerator in \eqref{Expansion-no-h}.
\begin{lemma}\label{lem:expans-num-3D} 
We have 
\begin{align*}
\int_\O |\n \Psi_\e|^2 dy+  \int_{\O} h  \Psi_\e^2 dy=&S_{3,\s}    - \e \textbf{m}(y_0) \textbf{c}^2\int_{\de \B_r }  \frac{\de {\cR}}{\de \nu} d\s(x)+\calO_r(\e) ,
\end{align*}
where $\nu$ is the unit outer normal of $\B_r$.

\end{lemma}
\begin{proof}
Recalling \eqref{eq:def-W-eps-3D}, 
direct computations give
\begin{align} \label{eq:nab-Psi}
\int_{F({Q}_{2r})\setminus F\left(\B_r\right)} |\n \Psi_\e|^2 dy&= \int_{F({Q}_{2r}) \setminus F\left(\B_r\right)} |\n \left( \ti \eta_r u_\e\right)|^2 dy+\e \textbf{c} ^2  \int_{F({Q}_{2r})\setminus F\left(\B_r\right)} |\n  \ti M_{2r} |^2 dy \nonumber\\
&+2 \e^{1/2}  \textbf{c}  \int_{F({Q}_{2r}) \setminus F\left(\B_r\right)} \n \left( \ti \eta_r u_\e\right) \cdot \n \ti M_{2r} dy \nonumber\\
 &= \e \int_{F({Q}_{2r}) \setminus F\left(\B_r\right)} |\n \left( \ti \eta_r V_\e\right)|^2 dy+\e  \textbf{c}^2   \int_{F({Q}_{2r}) \setminus F\left(\B_r\right)} |\n \ti  M_{2r} |^2 dy \nonumber\\
&+2 \e  \textbf{c}  \int_{F({Q}_{2r})\setminus F\left(\B_r\right)} \n \left( \ti \eta_r V_\e\right) \cdot \n \ti M_{2r} dy. 
\end{align}
By \eqref{eq:def-cut-off-cylind}, $\eta_r v_\e= \eta_r \e^{-1}w(\cdot/\e)$ is cylindrically symmetric. Therefore by the change   variable $y=F(x)$ and  using Lemma \ref{lem:to-prove}, we get
\begin{align}\label{eq:enev}
\e \int_{F({Q}_{2r}) \setminus F\left(\B_r\right)} |\n \left( \ti \eta_r V_\e\right)|^2 dy&= \e \int_{{Q}_{2r} \setminus  \B_r } |\n \left(  \eta_r {v}_\e\right)|^2_g \sqrt{g} dx \nonumber\\
&= \e \int_{{Q}_{2r} \setminus  \B_r} |\n \left(  \eta_r {v}_\e\right)|^2 dx+O\left(\e  r^2 \int_{{Q}_{2r} \setminus  \B_r} |\n \left(  \eta_r {v}_\e\right)|^2 dx \right).
\end{align}
By computing, we find that 
\begin{align*}
\e  \int_{{Q}_{2r} \setminus  \B_r} |\n \left(  \eta_r {v}_\e\right)|^2 dx& \leq \e \int_{{Q}_{2r} \setminus  \B_r}   |\n  {v}_\e |^2 dx+\e  \int_{{Q}_{2r} \setminus  \B_r} v_\e^2 |\n \eta_r  |^2 dx + 2 \e \int_{{Q}_{2r} \setminus  \B_r} v_\e  |\n v_\e|| \n\eta_r|   dx  \nonumber\\
& \leq \e\int_{{Q}_{2r} \setminus  \B_r}   |\n  {v}_\e |^2 dx+ \frac{C}{r^2} \e\int_{{Q}_{2r} \setminus  \B_r} v_\e^2  dx + \frac{C}{r}  \e \int_{{Q}_{2r} \setminus  \B_r} v_\e  |\n v_\e|   dx \nonumber\\
&=\int_{\B_{2r/\e} \setminus  \B_{r/\e}}   |\n w |^2 dx+C \frac{\e}{r^2} \int_{\B_{2r/\e} \setminus  \B_{r/\e}}w^2  dx + \frac{C}{r}  \e \int_{\B_{2r/\e} \setminus  \B_{r/\e}}  w  |\n w|   dx.
\end{align*}
From  this and \eqref{eq:est-nw-sq-3D} and \eqref{eq:est-wnw-sq-3D}, we get 
$$
O\left(\e  r^2 \int_{{Q}_{2r} \setminus  \B_r} |\n \left(  \eta_r {v}_\e\right)|^2 dx \right)= \calO_r(\e).
$$
We replace  this in \eqref{eq:enev} to  have 
\begin{align}\label{eq:int-nv-et-v}
\e \int_{F({Q}_{2r}) \setminus F\left(\B_r\right)} |\n \left( \ti \eta_r  V_\e\right)|^2 dy&= \e \int_{{Q}_{2r} \setminus  \B_r}   |\n (\eta_r {v}_\e) |^2 dx + \calO_r(\e).
\end{align}
We have the following estimates
\be \label{eq:est-nv-eps}
0\leq v_\e\leq C |x|^{-1}\quad  \textrm{ for $x\in \R^3\setminus\{0\}$ }\qquad  \textrm{  and  }\qquad |\n v_\e(x)| \leq C |x|^{-2} \quad\textrm{ for $|x|\geq \e$, }
\ee
which easily follows from \eqref{eq:up-low-bound-w-3D} and     Corollary \ref{cor:dec-est-w}.  
By these estimates,  Lemma \ref{MaMetricMetric} and \eqref{eq:regul-beta} together with  the change of variable $y=F(x)$, we have 
\begin{align*} 
\e \int_{F({Q}_{2r})\setminus F\left(\B_r\right)} \n \left( \ti \eta_r V_\e\right) \cdot \n \ti M_{2r} dy=& \e  \int_{ {Q}_{2r}\setminus  \B_{r} } \n \left(  \eta_{ {r} }  v_\e \right) \cdot \n   M   dx\\
&+ O\left( \e  \int_{ {\B}_{2r}\setminus  \B_{r} }|\n v_\e| dx   + \frac{\e}{r}   \int_{ {\B}_{2r}\setminus  \B_{r} } v_\e dx \right)\nonumber\\
= &\e  \int_{ {\B}_{2r}\setminus  \B_{r} } \n \left(  \eta_{ {r} }  v_\e \right) \cdot \n   M  dx + \calO_r(\e).
\end{align*}
This  with \eqref{eq:int-nv-et-v}, \eqref{eq:regul-beta} and \eqref{eq:nab-Psi}  give 
\begin{align*}
\int_{F({Q}_{2r})\setminus F\left(\B_r\right)} |\n \Psi_\e|^2 dy&= \e \int_{ {Q}_{2r} \setminus  \B_r } |\n \left(  \eta_r  {v}_\e\right)|^2 dx+\e  \textbf{c}^2   \int_{ {Q}_{2r}  \setminus  \B_r } |\n (\eta_{2r} M )|^2 dx\\
&+2 \e  \textbf{c}  \int_{{Q}_{2r}\setminus  \B_r } \n \left(  \eta_r  {v}_\e\right) \cdot \n  M  dx+ \calO_r(\e). 
\end{align*}  
Thanks  to Lemma \ref{lem:v-to-cR} and \eqref{eq:est-nv-eps}, we can thus use the dominated convergence theorem to deduce that, as $\e\to 0$, 
\be \label{eq:claim1-est-num}
\int_{ {Q}_{2r} \setminus  \B_r } |\n \left(  \eta_r  {v}_\e\right)|^2 dx= \textbf{c}^2\int_{ {Q}_{2r} \setminus  \B_r } |\n \left(  \eta_r  \cR\right)|^2 dx+o(1).
\ee
Similarly, we easily see that 
$$
 \int_{{Q}_{2r}\setminus  \B_r } \n \left(  \eta_r  {v}_\e\right) \cdot \n  M  dx= \textbf{c} \int_{{Q}_{2r}\setminus  \B_r } \n \left(  \eta_r  \cR \right) \cdot \n  M  dx+o(1)\qquad\textrm{ as $\e\to 0$.}
$$
This and  \eqref{eq:claim1-est-num}, then give
\begin{align}
\int_{F({Q}_{2r})\setminus F\left(\B_r\right)} |\n \Psi_\e|^2 dy 
&= \e\textbf{c}^2 \int_{ {Q}_{2r} \setminus  \B_r } |\n \left(  \eta_r  \cR\right)|^2 dx+\e  \textbf{c}^2   \int_{ {Q}_{2r}  \setminus  \B_r } |\n   M |^2 dx \nonumber\\
&+2 \e  \textbf{c}^2  \int_{{Q}_{2r}\setminus  \B_r } \n \left(  \eta_r  \cR \right) \cdot \n  M  dx+ \calO_r(\e) \nonumber\\
&=\e\textbf{c}^2 \int_{{Q}_{2r}\setminus  \B_r }  | \n ( \eta_r \cR+ M )|^2 dx+  \calO_r(\e). \label{eq:nPs-B2r-Br}
\end{align}
Since the support of $\Psi_\e$ is contained in $ \B_{4r}$ while the one of  $\eta_r$  is in ${Q}_{2r}$,  it is easy to deduce  from  \eqref{eq:regul-beta} that 
\begin{align*}
\int_{\O\setminus F\left({Q}_{2r}\right)} |\n \Psi_\e|^2 dy&= \e \textbf{c} ^2  \int_{F(\B_{4r})\setminus F\left({Q}_{2r}\right)} |\n  \ti M_{2r} |^2 dy=  \calO_r(\e)
\end{align*}
and from Lemma \ref{lem:estimates-for-3D}, that
$$
\int_{\O\setminus F\left(\B_r\right)}h | \Psi_\e|^2 dy=\e \textbf{c} ^2  \int_{F(\B_{4r})\setminus F\left(\B_{r}\right)} h | \eta_r V_\e +\ti M_{2r} |^2 dy=   \calO_r(\e).
$$
Therefore by \eqref{eq:nPs-B2r-Br}, we conclude that 
\begin{align*}
\int_{\O\setminus F\left(\B_r\right)} |\n \Psi_\e|^2 dy&+\int_{\O \setminus F\left(\B_r\right)}h | \Psi_\e|^2 dy \\
&\qquad=\e \textbf{c}^2  \int_{{Q}_{2r}\setminus  \B_r }  | \n ( \eta_r \cR+ M )|^2 dx+\e  \textbf{c}^2 \int_{{Q}_{2r}\setminus \B_r  }h(\cdot+y_0) | \eta_r \cR+ M  |^2 dx+ \calO_r(\e).
\end{align*}
Recall that $G(x+y_0,y_0)= \eta_r(x) \cR(x)+ M(x )$  for ever $x\in {Q}_{2r}$ and that  by \eqref{eq:Green-3D},   
 $$
 -\D_x G(x+y_0,y_0)+h(x+y_0) G(x+y_0,y_0)=0  \qquad \textrm{ for every  $x \in Q_{2r}\setminus Q_r $. }
 $$
 Therefore,   by integration by parts,   we find that
\begin{align*} 
\int_{\O\setminus F\left(\B_r\right)} |\n \Psi_\e|^2 dy+&\int_{\O\setminus F\left(\B_r\right)}h | \Psi_\e|^2 dy=\textbf{c}^2\int_{\de ({{Q}_{2r}\setminus \B_r )}} ( \eta_r \cR+ M ) \frac{\de ( \eta_r \cR+ M )}{\de \ov \nu}\s(x)+\calO_r(\e),
\end{align*} 
where $\ov \nu $  is the exterior normal vectorfield to ${Q}_{2r}\setminus \B_r  $. 
Thanks to   \eqref{eq:regul-beta}, we finally get 
\begin{align}\label{eq:nPsi-Ome-Br}
\int_{\O\setminus F\left(\B_r\right)} |\n \Psi_\e|^2 dy+&\int_{\O\setminus F\left(\B_r\right)}h | \Psi_\e|^2 dy 
&=-\e  \textbf{c}^2\int_{\de  \B_r } {\cR} \frac{\de {\cR}}{\de \nu} d\s(x)- \e \textbf{c}^2 \int_{\de  \B_r } M   \frac{\de {\cR}}{\de \nu} d\s(x) 
 + \calO_r(\e),
\end{align} 
where $\nu$ is  the exterior normal vectorfield to $  \B_r  $.\\
 Next we make the  expansion of $\int_{F\left(\B_r\right)} |\n \Psi_\e|^2 dy$ for $r$ and $\e$ small. First, we observe that, by  Lemma \ref{lem:estimates-for-3D} and \eqref{eq:regul-beta}, we have 
\begin{align*}
\int_{F\left(\B_r\right)} |\n \Psi_\e|^2 dy &=\int_{F\left(\B_r\right)} |\n u_\e|^2 dy+\e \textbf{c}^2 \int_{F\left(\B_r\right)} |\n M |^2 dy+2\e^{1/2}\textbf{c} \int_{F\left(\B_r\right)} \n u_\e \cdot \n \ti M_{2r} dy\\
&= \int_{ \B_{r/\e}} |\n  w|^2 dx +O\left( \e^2 \int_{ \B_{r/\e}} |x|^2 |\n  w|^2 dx +\e^{2}   \int_{   \B_{r/\e} } |\n  w |   dx \right) +\calO_r(\e)\\
&=  \int_{ \B_{r/\e}} |\n  w|^2 dx+ \calO_r(\e).
\end{align*}
By integration by parts and using \eqref{eq:est-L2-star-3D}, we deduce that 
\begin{align}\label{eq:expans-num-FQr}
\int_{F\left(\B_r\right)} |\n \Psi_\e|^2 dy 
&= S_{3,\s}  \int_{ \B_{r/\e}} |z|^{-\s} w^{2^*_\s}   dx+ \int_{\de  \B_{r/\e}} w\frac{\de w }{\de \nu } d\s(x)+\calO_r(\e) \nonumber\\
&=  S_{3,\s}  + \e\int_{\de  \B_{r}} v_\e\frac{\de v_\e }{\de \nu } d\s(x) + \calO_r(\e). 
\end{align}
Now  \eqref{eq:est-nv-eps},   \eqref{eq:nv-eps-to-nv-C1} and the dominated convergence theorem yield, for  fixed $r>0$ and $\e\to 0$,
\begin{align}\label{eq:d-eps-nv-eps}
\int_{\de  \B_{r}} v_\e\frac{\de v_\e }{\de \nu } d\s(x)&=\int_{\de B^2_{\R^2}(0,r)}\int_{-r}^r v_\e(t,z)\n v_\e(t,z)\cdot\frac{z}{|z|} d\s(z) dt+2\int_{B^2_{\R^2}} v_\e(r,z)  \de _tv_\e(r,z) dz \nonumber\\
&= \textbf{c}^2\int_{\de B^2_{\R^2}(0,r)}\int_{-r}^r \cR(t,z)\n \cR(t,z)\cdot\frac{z}{|z|} d\s(z) dt+2 \textbf{c}^2\int_{B^2_{\R^2}} \cR(r,z) \de _t \cR(r,z) dz+o(1)\nonumber\\
& = \textbf{c}^2\int_{\de  \B_r } {\cR} \frac{\de {\cR}}{\de \nu} d\s(x) + o(1).
\end{align}
Moreover  \eqref{eq:est-nw-3D} implies that
\begin{align*}
  \int_{F(\B_{r})} h  \Psi_\e^2 dy  = \calO_r(\e).
\end{align*} 
From this together with \eqref{eq:expans-num-FQr} and \eqref{eq:d-eps-nv-eps}, we obtain
\begin{align*}
\int_{F\left(\B_r\right)} |\n \Psi_\e|^2 dy +  \int_{F(\B_{r})} h  \Psi_\e^2 dy
&=  S_{3,\s}  +  \textbf{c}^2\e\int_{\de  \B_r } {\cR} \frac{\de {\cR}}{\de \nu} d\s(x)  +\calO_r(\e). 
\end{align*}
Combining this with \eqref{eq:nPsi-Ome-Br}, we then  have 
\begin{align} \label{eq:Dirichlet-Psi-eps}
\int_\O |\n \Psi_\e|^2 dy+  \int_{\O} h  \Psi_\e^2 dy=&S_{3,\s}    - \e \textbf{c}^2\int_{\de \B_r } M \frac{\de {\cR}}{\de \nu} d\s(x)+\calO_r(\e) +o\left(\e\right).
\end{align}
Since (recalling \eqref{eq:def-mass}) $M(y)=M(0)+O(r)=\textbf{m}(y_0)+O(r)$ in $\B_{2r}$, we get the claimed result in the statement of the lemma.
\end{proof}

The following result together with the previous lemma provides the proof of Proposition\ref{Expansion-no-h}.

\begin{lemma}\label{lem:expans-denom-3D}
We have
$$
\begin{array}{ll}
\displaystyle \left( \int_{\O} \rho^{-\s}_\G |\Psi_{\e}|^{2^*_\s} dy \right)^{\frac{2}{2^*_\s}}&  =
   \displaystyle 1
- \frac{2}{S_{3,\s}} {\e} \textbf{m}(y_0) \textbf{c}^2   \int_{\de \B_{r}} \frac{\de \cR}{\de \nu} d\s(x)  +\calO_r(\e).
\end{array}
$$
\end{lemma}

\begin{proof}

 Since $2^*_\s>2$, there exists a positive constant $C(\s)$  such that
\begin{equation*}
||a+b|^{2^*_\s}-|a|^{2^*_\s}-2^*_\s ab |a|^{2^*_\s-2}| \leq C(\s) \left(|a|^{2^*_\s-2} b^2+|b|^{2^*_\s}\right)\qquad\textrm{  for all $a,b \in \R$.}
\end{equation*}
As a consequence, we obtain
\begin{align}\label{eq:expan-L-2-star-1}
\displaystyle\int_{\O} &\rho^{-\s}_\G |\Psi_{\e}|^{2^*_\s} dy= \displaystyle \int_{F(\B_{r})} \rho^{-\s}_\G |u_{\e}+ \e^{\frac{1}{2}} \ti {M}_{2r}|^{2^*_\s} dy+ \int_{F(\B_{4r}) \setminus F(\B_{r})} \rho^{-\s}_\G |W_{\e}+ \e^{\frac{1}{2}} \ti {M}_{2r}|^{2^*_\s} dy \nonumber\\
&=\displaystyle \int_{F(\B_{r})} \rho^{-\s}_\G | u_{\e}|^{2^*_\s} dy
+
2^*_\s \textbf{c} {\e}^{1/2} \int_{F(\B_{r})} \rho^{-\s}_\G | u_{\e}|^{2^*_\s-1}   \ti M_{2r} dy \nonumber\\
&\quad  \displaystyle + O\left(\int_{F\left(\B_{ 4r}\right)}   \rho^{-\s}_\G |\eta_r u_{\e}|^{2^*_\s-2} \left({\e}^{1/2}\ti {M}_{2r} \right)^2 dy
+ \int_{F\left(\B_{ 4r}\right)} \rho^{-\s}_\G |{\e}^{1/2} \ti {M}_{2r} |^{2^*_\s} dy\right) \nonumber\\
&\quad  \displaystyle+O \left(  \int_{F(\B_{4r}) \setminus F(\B_{r})} \rho^{-\s}_\G | u_{\e}|^{2^*_\s} dy
+
2^*_\s \textbf{c} {\e}^{1/2} \int_{F(\B_{4r}) \setminus F(\B_{r})} \rho^{-\s}_\G | u_{\e}|^{2^*_\s-1}   \ti {M}_{2r} dy   \right).
\end{align}
By H\"{o}lder's inequality and \eqref{DeterminantMetric}, we have
\begin{align}\label{eq:expan-L-2-star-2}
\int_{F\left(\B_{ 4r}\right)}   \rho^{-\s}_\G |\eta u_{\e}|^{2^*_\s-2} \left({\e}^{1/2}\ti \b_r \right)^2 dy& \leq \e \|u_{\e}\|_{L^{2^*_\s}(F(\B_{4 r} );\rho^{-\s})}^{{2^*_\s-2}}\|\ti{M}_{2r}\|_{L^{2^*_\s}(F(\B_{4 r} );\rho^{-\s}_\G)}^{{2} }\nonumber\\
&=\e \|w\|_{L^{2^*_\s}( \B_{ 4r};|z|^{-\s}\sqrt{|g|} )}^{{2^*_\s-2}}  \|\ti{M}_{2r}\|_{L^{2^*_\s}(F(\B_{4 r} );\rho^{-\s}_\G)}^{{2} }\nonumber\\
&\leq \e (1+ Cr)\|\ti{M}_{2r}\|_{L^{2^*_\s}(F(\B_{ 4r} );\rho^{-\s}_\G)}^{{2} }=\calO_r(\e),
\end{align}
recalling that $ \|w\|_{L^{2^*_\s}( \R^3;|z|^{-\s} )}=1 $.
Furthermore, since $2^*_\s>2$, by \eqref{eq:regul-beta}, we easily get 
\begin{align} \label{eq:expan-L-2-star-2-00}
\int_{F\left(\B_{4 r}\right)} \rho^{-\s}_\G |{\e}^{1/2} \ti {M}_{2r} |^{2^*_\s} dy=o(\e).
\end{align}
Moreover  by change of variables and \eqref{eq:est-L2-star-3D}, we also have
\begin{align*}
\int_{F(\B_{4r}) \setminus F(\B_{r})} \rho^{-\s}_\G | u_{\e}|^{2^*_\s} dy&
+
2^*_\s \textbf{c} {\e}^{1/2} \int_{F(\B_{4r}) \setminus F(\B_{r})} \rho^{-\s}_\G | u_{\e}|^{2^*_\s-1}   \ti M_{2r} dy\\
&     \leq C \int_{\B_{4r/\e} \setminus  \B_{r/\e}} |z|^{-\s}  | w|^{2^*_\s} dx
+
C   {\e}   \int_{\B_{4r/\e} \setminus  \B_{r/\e} } |z|^{-\s}  | w|^{2^*_\s-1}    dx\\
&  =o(\e).
\end{align*}
By this, \eqref{eq:expan-L-2-star-1}, \eqref{eq:expan-L-2-star-2-00} and \eqref{eq:expan-L-2-star-2}, it results
\begin{align*}
\displaystyle\int_{\O} \rho^{-\s}_\G |\Psi_{\e}|^{2^*_\s} dy&=  \displaystyle \int_{F(\B_{r})} \rho^{-\s}_\G | u_{\e}|^{2^*_\s} dy
+
2^*_\s \textbf{c} {\e}^{1/2} \int_{F(\B_{r})} \rho^{-\s}_\G | u_{\e}|^{2^*_\s-1}   \ti M_{2r} dy +\calO_r(\e).
\end{align*}
We define  $B_\e(x):=M(\e x)  \sqrt{|g_\e|}(x) =M(\e x)  \sqrt{|g|}(\e x)$. Then
by    the change of variable $y=\frac{F(x)}{\e}$ in the above identity and recalling \eqref{DeterminantMetric}, then by oddness, we have 
\begin{align*}
\displaystyle\int_{\O} \rho^{-\s}_\G |\Psi_{\e}|^{2^*_\s} dy &=  \displaystyle\int_{\B_{r/\e}} |z|^{-\s}  w ^{2^*_\s} \sqrt{|g_\e|}dx
+
2^*_\s {\e}  \textbf{c}  \int_{\B_{r/\e}}|z|^{-\s} | w |^{2^*_\s-1} B_\e dx+ \calO_r(\e)\\
&= \displaystyle\int_{\B_{r/\e}} |z|^{-\s}  w ^{2^*_\s} dx
+
2^*_\s {\e}  \textbf{c}  \int_{\B_{r/\e}}|z|^{-\s} | w |^{2^*_\s-1} B_\e dx+ \calO_r(\e)\\
&\quad \displaystyle+  O\left(   \e^2\int_{\B_{r/\e}} |z|^{-\s} |x|^2  w ^{2^*_\s}  dx\right)\\
&=\displaystyle 1 + 2^*_\s {\e}  \textbf{c}  \int_{\B_{r/\e}}|z|^{-\s} | w |^{2^*_\s-1} B_\e dx\\
&\quad \displaystyle+O\left( \int_{\R^3\setminus \B_{r/\e}} |z|^{-\s}  w ^{2^*_\s}  dx+ \e^2\int_{\B_{r/\e}} |z|^{-\s} |x|^2  w ^{2^*_\s}  dx\right)+ \calO_r(\e).
\end{align*}
  Therefore  by \eqref{eq:est-L2-star-3D} we then have 
\be \label{eq:est-L-2star-Psi-not-ok}
\begin{array}{ll}
\displaystyle \left( \int_{\O} \rho^{-\s}_\G |\Psi_{\e}|^{2^*_\s} dy \right)^{\frac{2}{2^*_\s}}& =
   \displaystyle 1
+
2 {\e}  \textbf{c}  \int_{\B_{r/\e}}|z|^{-\s} | w |^{2^*_\s-1} B_\e(x)  dx +\calO_r(\e).
\end{array}
\ee
Multiply \eqref{eq:w-gorund-3D} by $B_\e\in \cC^1(\overline{Q_r})$ and integrate by parts to get 
\begin{align*}
S_{3,\s} \int_{\B_{r/\e}}|z|^{-\s} | w |^{2^*_\s-1}  B_\e    dx& = \int_{ \B_{r/\e} } \n w \cdot \n B_\e dx -\int_{ \de \B_{r/\e} }B_\e  \frac{ \de  w}{\de \nu}  d\s(x)\\
&= \int_{ \B_{r/\e} } \n w \cdot \n B_\e dx-\int_{ \de \B_{r} }B_1  \frac{ \de  v_\e}{\de \nu}  d\s(x).
\end{align*}
Since $|\n B_\e|\leq C \e$,  by Lemma \ref{lem:v-to-cR} and \eqref{eq:regul-beta}, we then have 
$$
\e   \int_{ \B_{r/\e} } \n w \cdot \n B_\e dx  =O\left( \e^2  \int_{ \B_{r/\e} } |\n w| dx \right)= \calO_r(\e). 
$$
Consequently
\begin{align*}
S_{3,\s} \e  \int_{\B_{r/\e}}|z|^{-\s} | w |^{2^*_\s-1}  B_\e    dx 
&=  -\e\int_{ \de \B_{r} }B_1  \frac{ \de  v_\e}{\de \nu}  d\s(x)+ \calO_r(\e),
\end{align*}
on the one hand.
On the other hand    by Lemma \ref{lem:v-to-cR}, \eqref{eq:regul-beta}  and the dominated convergence theorem, we get
\begin{align*}
 \int_{ \de \B_{r} } B_1  \frac{\de  v_\e}{\de \nu}    d\s(x)= \textbf{c}  \int_{ \de \B_{r} } B_1  \frac{\de  \cR}{\de \nu}    d\s(x)+o(1)= \textbf{c} {M}(0) \int_{ \de \B_{r} }    \frac{\de  \cR}{\de \nu}    d\s(x)+O(r)+o(1),
\end{align*}
so that 
\begin{align*}
 \e \textbf{c}  \int_{\B_{r/\e}}|z|^{-\s} | w |^{2^*_\s-1}  B_\e    dx 
&=  -\e  \textbf{c}^2\frac{1 }{S_{3,\s}}M(0)\int_{ \de \B_{r} }   \frac{ \de  \cR}{\de \nu}  d\s(x)+ \calO_r(\e).
\end{align*}
It then follows from \eqref{eq:est-L-2star-Psi-not-ok} that
$$
\begin{array}{ll}
\displaystyle \left( \int_{\O} \rho^{-\s}_\G |\Psi_{\e}|^{2^*_\s} dy \right)^{\frac{2}{2^*_\s}}& =
   \displaystyle 1 -\frac{2}{S_{3,\s}} {\e}\textbf{c}^2  {M}(0)  \int_{ \de \B_{r} }  \frac{\de  \cR}{\de \nu}    d\s(x)+\calO_r(\e).
\end{array}
$$
Since ${M} (0)=\textbf{m}(y_0)$, see \eqref{eq:def-mass}, the proof of the lemma is thus finished.
\end{proof}

\begin{proof}[Proof of Proposition \ref{Expansion-no-h} (completed)]
By Lemma \ref{lem:expans-num-3D} and Lemma \ref{lem:expans-denom-3D}, we have 
 \begin{align}\label{eq:est-Quotient-3D}
\displaystyle J(\Psi_\e) 
& =  S_{3,\s}-  \e    \textbf{c}^2  \textbf{m}(y_0)  \int_{\de \B_r}\frac{\de \cR}{\de \nu}\, d\s(x)    +\calO_r(\e)  .
\end{align}

Finally, recalling that $\cR(x)=\frac{1}{|x|}$,  we  can compute 
\begin{align*}
\int_{\de \B_r}\frac{\de \cR}{\de \nu}\, d\s(x)&=-  \int_{\de \B_r}\frac{x\cdot \nu (x)}{|x|^3}\, d\s(x)\\
&=- 2 r   \int_{ B_{\R^2}(0, r)}\frac{1 }{r^2+|z|^2}\, dz - 2\pi\int_{-r}^r\frac{r^3}{r^2+t^2}dt\\
&=- \pi^2 (1+ r^2)   .
\end{align*}
From this and \eqref{eq:est-Quotient-3D}, we then  have 
 \begin{align*}
\displaystyle J(\Psi_\e) 
&= S_{3,\s}-  \e \pi^2  \textbf{c}^2 \textbf{m}(y_0)   +\calO_r(\e)  ,
\end{align*}
which finishes the proof.
\end{proof}
As a consequence of   Proposition \ref{Expansion-no-h}, we can now complete the
\begin{proof}[Proof of Theorem \ref{th:main2} (completed)]
The proof  follows from Lemma \ref{Expansion-no-h} that if $\textbf{m}(y_0)>0$ for some $y_0 \in \G$ then $\mu_h(\O,\G)< S_{3,\s}$. The fact that $\mu_h(\O,\G)$ is attained by a positive function is an immediate consequence of Proposition \ref{Pro3333}  belwo.
\end{proof}

\section{Appendix: Existence of minimizer for $\mu_h\left(\O,\G\right)$}\label{s:Section4}
 Let $\O$ be a bounded domain in $\R^N$, $N\geq 3$, and $h$ a continuous function on $\O$. Let  $\G$ be a smooth closed curve $\G  $ contained in $\O$.  We consider 
\be \label{eq:mu-h-O-G-appendix}
\mu_{h}(\O,\G):=\inf_{ u\in  H^1_0(\O)  }\frac{ \displaystyle \int_{\O} |\nabla u|^2 dx +\int_\O h u^2 dy }{  \displaystyle \left(  \int_{\O} \rho_\G^{-\s} |u|^{2^*_\s} dy \right)^{\frac{2}{2^*_\s}}}.
\ee
We also recall  that
\be\label{eq:S-N-sig-appendix}
S_{N,\s}=\inf_{  v\in \cD^{1,2}(\R^N)  }\frac{ \displaystyle  \int_{\R^N} |\nabla v|^2 dx}{   \displaystyle \left(  \int_{\R^N} |z|^{-\s} |v|^{2^*_\s} dx \right)^{\frac{2}{2^*_\s}} },
%
\ee
with $x=(t,z)\in \R\times \R^{N-1}$.
  Our aim in this section is to show   that   { if } $\mu_h\left(\O,\G \right)< S_{N,\s}$   then the best constant  $\mu_h\left(\O,\G\right)$    is achieved. The argument of proof is standard. However, for sake of completeness, we add the proof. We start with the following
\begin{lemma}\label{OuiOui}
Let $\O$ be an open subset of $\R^N$, with $N\geq 3$, and let $\G\subset \O$ be a closed smooth curve.
Then for every $r>0$, there exist  positive constants $  c_r>0$, only  depending on $\O,\G,N,\s$ and $r$, such that  for every  $u \in H^1_0(\O)$  
$$
S_{N,\s} \left(\int_{\O} \rho^{-\s}_\G |u|^{2^*_\s} dy\right)^{2/2^*_\s} \leq(1+ r) \int_{\O} |\n u|^2dy+c_r \left[\int_\O u^2 dy+\left(\int_\O |u|^{2^*_\s} dy\right)^{2/2^*_\s}\right],
$$
where $2^*_\s=\frac{2(N-\s)}{N-2}$ and $\s\in (0, 2)$.
\end{lemma}
\begin{proof}
We let $r>0$ small. We can cover a tubular neighborhood of $\G$ by a finite number of  sets $\left(T^{y_i}_r\right)_{1\leq i \leq m}$  given by 
$$
T^{y_i}_r:=F_{y_i}\left(Q_r\right), \qquad \textrm{ with $y_i\in\G$. }
$$
We refer to  Section \ref{s:Geometric-prem} for the parameterization  $F_{y_i}:Q_r\to \O$. 
See e.g. \cite[Section 2.27]{Aubin-Book}, there exists  $\left(\vp_i\right)_{1\leq i \leq m}$  a   partition of unity subordinated to this covering such that
\begin{equation}\label{eq:Part-unity}
\sum_i^{m } \vp_i  =1 \qquad \textrm{and} \qquad |\n \vp_i ^{\frac{1}{2^*_\s}} |\leq K \qquad \textrm{ in } U:=\displaystyle \cup_{i=1}^{m} T^{y_i}_r,
\end{equation}
for some constant $K>0$.
We define
\be \label{eq:def-psi_i}
\psi_i(y):=\vp_i^{\frac{1}{2^*_\s}} (y) u(y) \qquad \textrm{ and } \qquad \ti{\psi_i}(x)=\psi_i (F_{y_i}(x)).
\ee
Recall that that $\rho_\G\geq C>0$ on  $\O \setminus U$, for some positive constant $C>0$. Therefore,  since $\frac{2}{2^*_\s}<1$, by \eqref{eq:def-psi_i} we get
\begin{align} \label{eq:Sum-append}
\displaystyle\left(\int_\O \rho^{-\s}_\G |u|^{2^*_\s} dy\right)^{2/2^*_\s} 
&\displaystyle\leq   \left(\int_U \rho^{-\s}_\G \left| u \right|^{2^*_\s} dy\right)^{2/2^*_\s}+ \left(\int_{\O\setminus U} \rho^{-\s}_\G |u|^{2^*_\s} dy\right)^{2/2^*_\s}\nonumber\\
&\displaystyle \leq \left( \sum_i^{m} \int_{T^{y_i}_r} \rho^{-\s}_\G  |\psi_i|^{2^*_\s} dy\right)^{2/2^*_\s}+c_r  \left(\int_\O |u|^{2^*_\s} dy\right)^{2/2^*_\s} \nonumber\\
&\displaystyle \leq\sum_i^{m} \left(  \int_{T^{y_i}_r} \rho^{-\s}_\G  |\psi_i|^{2^*_\s} dy\right)^{2/2^*_\s}+c_r  \left(\int_\O |u|^{2^*_\s} dy\right)^{2/2^*_\s}
\end{align}
By  change of variables and Lemma \ref{MaMetricMetric}, we have
\begin{align*}
\left(\int_{T^{y_i}_r} \rho^{-\s}_\G |\psi_i|^{2^*_\s} dy\right)^{2/2^*_\s} &=\left(\int_{Q_r} |z|^{-\s} |\ti{\psi}_i|^{2^*_\s} \sqrt{|g|}(x) dx\right)^{2/2^*_\s}\\
&\leq \left(1+c r\right) \left(\int_{Q_r} |z|^{-\s} |\ti{\psi}_i|^{2^*_\s}  dx\right)^{2/2^*_\s}.
\end{align*}
In addition the Hardy-Sobolev inequality \eqref{eq:Hard-Sob-sharp} yields
$$
S_{N,\s}  \left(\int_{Q_r} |z|^{-\s} |\ti{\psi}_i|^{2^*_\s}  dx\right)^{2/2^*_\s}\leq  \left( \int_{Q_r} |\n \ti{\psi}_i|^2 dx \right)^{2/2}.
$$
Therefore  by change of variables and  Lemma \ref{MaMetricMetric},  we get
\begin{align*}
S_{N,\s}&\left(\int_{T^{y_i}_r} \rho^{-\s}_\G |\psi_i|^{2^*_\s} dy\right)^{2/2^*_\s}  \leq \left(1+cr\right)\int_{Q_r} |\n \ti{\psi}_i|^2 dx\\
&\leq \left(1+ c' r\right) \int_{T^{y_i}_r} |\n (\phi_i^{\frac{1}{2^*_\s}}  u)|^2 dy=  \left(1+ c'r\right) \int_{T^{y_i}_r}|\phi_i^{\frac{1}{2^*_\s}}\n u+ u \n \phi_i^{\frac{1}{2^*_\s}}|^2 dy+ c_r  \int_\O | u|^2 dy.
\end{align*}
Applying Young's inequality   using \eqref{eq:Part-unity} and \eqref{eq:def-psi_i}, we find that
\begin{align*}
S_{N,\s}\left(\int_{T^{y_i}_r} \rho^{-\s}_\G |\psi_i|^{2^*_\s} dy\right)^{2/2^*_\s}& \leq   \left(1+ c' r\right)(1+\e) \int_{T^{y_i}_r} \phi_i^{\frac{2}{2^*_\s}}|\n u|^2 dy+ c_r(\e)  \int_\O | u|^2 dy\\
&\leq   \left(1+ c' r\right)(1+\e) \int_{T^{y_i}_r} |\n u|^2 dy+ c_r(\e)  \int_\O | u|^2 dy.
\end{align*}
Summing for $i$ equal $1$ to $m$,  we get 
\begin{align*}
S_{N,\s}  \sum_{i=1}^{m}\left(\int_{T^{y_i}_r} \rho^{-\s}_\G |\psi_i|^{2^*_\s} dy\right)^{2/2^*_\s}\leq     \left(1+ c' r\right)(1+\e)    \left( \int_{\O} |\n u|^2 dy\right)^{1/2}+ c_r(\e)  \left(\int_\O | u|^2 dy\right)^{1/2}.
\end{align*}
This together with \eqref{eq:Sum-append} give
$$
S_{N,\s} \left(\int_\O \rho^{-\s}_\G |u|^{2^*_\s} dy\right)^{2/2^*_\s}\leq  \left(1+ c' r\right)(1+\e)   \int_{\O} |\n u|^2 dy + c_r(\e)  \int_\O | u|^2 dy +c_r \left(\ \int_\O |u|^{2^*_\s} dy\right)^{2/2^*_\s}.
$$
Since $\e$ and $r$ can be  chosen arbitrarily small, we get the desired result.
\end{proof}
We can now prove the following existence result.
\begin{proposition}\label{Pro3333}
Consider $\mu_h(\O,\G)$ and $S_{N,\s}$ given by \eqref{eq:mu-h-O-G-appendix} and \eqref{eq:S-N-sig-appendix} respectively.
Suppose that 
\begin{equation}\label{Assumption}
\mu_h\left(\O,\G\right)<S_{N,\s}.
\end{equation}
Then $\mu_h\left(\O,\G\right) $ is achieved by a positive function.
\end{proposition}
\begin{proof}
Let $\left(u_n\right)_{n\in \N}$ be a minimizing sequence for $\mu_h\left(\O,\G\right)$ normalized so that
\begin{equation}\label{Minimization}
\int_\O \rho^{-\s}_\G |u|^{2^*_\s} dx=1 \quad \textrm{ and } \quad \mu_h\left(\O,\G\right)=\int_\O |\n u_n|^2 dx +\int_\O h u_n^2 dx+o(1).
\end{equation}
By coercivity of $-\D+ h$, the sequence $\left(u_n\right)_{n\in \N}$ is bounded in $H^1_0(\O)$ and thus , up to a subsequence, 
$$
u_n \rightharpoonup u \qquad\textrm{weakly in } H^1_0(\O),
$$
and
\be \label{eq:-un-to-u-strong}
u_n \to  u \qquad \textrm{ strongly in } L^p(\O)\quad \textrm{for}\quad 1\leq p < 2^*_0:=\frac{2N}{N-2}. 
\ee
The weak convergence  in  $H^1_0(\O)$ implies that
\begin{equation}\label{SB1}
\int_\O |\n u_n|^2 dx= \int_\O |\n(u_n-u)|^2 dx+\int_\O |\n u|^2 dx+o(1).
\end{equation}
By Brezis-Lieb lemma \cite{BL} and the strong convergence in the Lebesgue spaces $L^p(\O)$, we have
\begin{equation}\label{SB2}
1=\int_\O \rho^{-\s}_\G |u_n|^{2^*_\s} dx=\int_\O \rho^{-\s}_\G|u-u_n|^{2^*_\s} dx+ \int_\O \rho^{-\s}_\G |u|^{2^*_\s} dx+o(1).
\end{equation}
By Lemma \ref{OuiOui}, \eqref{eq:-un-to-u-strong} ---note that $2^*_\s< 2^*_0$ , we then deduce that
\begin{equation}\label{SB3}
S_{N,\s} \left(\int_\O \rho^{-\s}_\G |u-u_n|^{2^*_\s} dx\right)^{2/2^*_\s} \leq(1+r) \int_\O|\n (u-u_n)|^2 dx+o(1).
\end{equation}
Using \eqref{SB1}, \eqref{SB2} and \eqref{SB3}, we have
\begin{align}\label{eq:near-exist}
\displaystyle S_{N,\s}\left(1-\int_\O \rho^{-\s}_\G |u|^{2^*_\s} dx\right)^{2/2^*_\s} &\displaystyle\leq \left(1+r\right) \left(\int_\O |\n u_n|^2 dx -\int_\O |\n u|^2 dx\right)+o(1) \nonumber\\
&\displaystyle= \left(1+r\right) \left(\mu_h\left(\O,\G\right)-\int_\O h u_n^2 dx -\int_\O |\n u|^2 dx\right)+o(1) \nonumber\\
&\displaystyle= \left(1+r\right) \left(\mu_h\left(\O,\G\right)-\int_\O h u^2 dx -\int_\O |\n u|^2 dx\right)+o(1) \nonumber\\
&\displaystyle\leq \left(1+r\right) \mu_h\left(\O,\G\right)\left(1-\left(\int_\O \rho^{-\s}_\G |u|^{2^*_\s} dx\right)^{2/2^*_\s}\right)+o(1).
\end{align}
By  the concavity of the map $t\mapsto t^{2/2^*_\s}$ on $[0,1]$,  we have 
$$
 1 \leq \left(1-\int_\O \rho^{-\s}_\G |u|^{2^*_\s} dx\right)^{2/2^*_\s}+ \left(\int_\O \rho^{-\s}_\G |u|^{2^*_\s} dx\right)^{2/2^*_\s}.
$$
From this,  
then taking the limits respectively as $n \to +\infty$ and as $r \to 0$ in \eqref{eq:near-exist}, we find that
$$
\left[S_{N,\s}-\mu_h(\O,\G) \right]\left(1-\left(\int_\O \rho^{-\s}_\G |u|^{2^*_\s} dx\right)^{2/2^*_\s}\right)\leq 0.
$$
Thanks to  \eqref{Assumption},  we then  get 
$$
1 \leq \int_\O \rho^{-\s}_\G |u|^{2^*_\s} dx .
$$
Since by \eqref{Minimization}  and Fatou's lemma,
$$
1=\int_\O \rho^{-\s}_\G |u_n|^{2^*_\s} dx \geq \int_{\O} \rho^{-\s}_\G |u|^{2^*_\s} dx,
$$
we conclude that
$$
\int_\O \rho^{-\s}_\G |u|^{2^*_\s} dx=1.
$$
It then follows from \eqref{Minimization} that $u_n\to u$ in $L^{2^*_\s}(\O;\rho_\G^{-\s})$ and  thus $u_n\to u$ in $H^1_0(\O)$.
Therefore $u$ is a minimizer for $\mu_h(\O,\G) $.  Since $|u|$  is also a minimizer for $\mu_h(\O,\G)$, we may assume that $u \gneqq0$. Therefore $u>0$ by the maximum principle.
\end{proof}

\end{document}